\documentclass[16pt]{article}
\usepackage{amsmath,amssymb,amsfonts,amsthm,amsbsy}
\usepackage{graphicx,cite}
\usepackage{pstricks}
\usepackage{psfrag}
\usepackage{color}
\usepackage{subfigure}
\usepackage{nextpage}
\usepackage{layout}
\usepackage{amsthm}
\usepackage{dsfont}
\usepackage{amssymb}
\usepackage{subfigure}
\usepackage{epstopdf}

\usepackage[top=1.5in,bottom=1in,left=1in,right=1in]{geometry}

\newtheorem{theorem}{Theorem}[section]

\newtheorem{lemma}[theorem]{Lemma}
\newtheorem{remark}[theorem]{Remark}

\newcommand{\argmin}{\mathop{\rm arg\min}}

\newcommand{\nn}{\nonumber}

\newcommand{\innerp}[1]{\langle {#1} \rangle}

\newcommand{\abs}[1]{\lvert#1\rvert}

\def\R{\mathbb{ R}}
\def\N{\mathbb{ N}}
\def\E{\mathbb{ E}}

\begin{document}

\title{On discrete least square projection in unbounded domain with
random evaluations and its application to parametric uncertainty quantification}

\author{
Tao Tang\!\!
\thanks{Department of Mathematics, The Hong Kong
Baptist University, Kowloon Tong, Kowloon, Hong Kong, China,
Email: {ttang@math.hkbu.edu.hk}}
\qquad
Tao Zhou\!\!
\thanks{Institute of Computational Mathematics and
Scientific/Engineering Computing, AMSS, the Chinese Academy of Sciences, Beijing, China,
Email: {tzhou@lsec.cc.ac.cn}}
}

\date{}

\maketitle

\begin{abstract}
This work is concerned with
approximating multivariate functions in unbounded domain by using
discrete least-squares projection with random points evaluations.
Particular attention are given to  functions with random Gaussian
or Gamma parameters. We first demonstrate that the traditional Hermite (Laguerre)
polynomials chaos expansion suffers from the \textit{instability}
in the sense that an \textit{unfeasible} number of points,
which is relevant to the dimension of the approximation space,
is needed to guarantee the stability in the least square framework.
We then  propose to use the Hermite/Laguerre {\em functions}
 (rather than polynomials) as bases in the expansion. The corresponding design points are obtained by
mapping the uniformly distributed random points in bounded intervals to the unbounded domain, which involved a mapping parameter $L$. By using the Hermite/Laguerre {\em functions}
and a proper mapping parameter,
the stability can be significantly improved even if
the number of design points scales \textit{linearly} (up to a
logarithmic factor) with the dimension of the approximation space.
Apart from the stability, another important issue is the rate of
convergence. To speed up the convergence, an effective scaling factor is introduced, and a principle for
choosing quasi-optimal scaling factor is discussed.
Applications to
parametric uncertainty quantification are illustrated by considering a
random ODE model together with an elliptic problem with lognormal random input.
\end{abstract}

\section{Introduction}
\setcounter{equation}{0}

In recent years, there has been a growing need to model uncertainty in mathematical and physical models and to quantify the resulting effect on output quantities of interest (QoI). Several methodologies for accomplishing these tasks fall under the growing sub-discipline of Uncertainty Quantification (UQ). In general, one can use a probabilistic setting to include these uncertainties in mathematical models. In such a framework, the random input parameters are modeled as random variables; infinite-dimensional analogues leveraging random fields with a prescribed correlation structure extend this procedure to more general settings. Frequently, the goal of this mathematical and computational analysis becomes the prediction of statistical moments of the solution, or statistics of some QoI, given the probability distribution of the input random data.

A fundamental problems in UQ is approximation of a multivariate function $Z=f(x,\mathbf{y})$ where the parameters
$\mathbf{y}=(y_1,y_2,...,y_d)$ are $d$-dimensional random vectors. The function $Z$ might be a solution resulting from a stochastic PDE problem or a derived QoI from such a system. Efficient and robust numerical methods that address such problems have been investigated in detail in recent years (see, e.g. \cite{Hyperbolic1,XiuK1,XiuK2,Ghanem,Xiu,XiuH,FabioC2} and references therein). One of these methods that have enjoyed much attention and success is the generalized Polynomial Chaos (gPC) method,
see, e.g.,  \cite{Xiu,XiuK1,XiuK2,Ghanem}, which is a generalization of the Wiener-Hermite polynomial chaos expansion \cite{Wiener}. In gPC, we expand the solution $Z$ in polynomials of the input random variables $y_i$. When $Z$ exhibits regular variation with respect to $y_i$, gPC yields efficient convergence rates with respect to the polynomial degree of expansion. With intrusive gPC approaches, existing deterministic solvers must be rewritten, and solvers for a coupled
system of deterministic equations are needed, which can be very complicated if the underlying differential equations have nontrivial
nonlinear form, see, e.g.,  \cite{XiuK1,Hyperbolic1,Hyperbolic2}.
By contrast, non-intrusive methods build a polynomial approximation
by leveraging only existing deterministic solvers in a
Monte-Carlo-like fashion.

To efficiently build a gPC approximation, one can resort to the discrete least-squares projection onto a polynomial space.  A major design criterion for this approach is the specification of $\mathbf{y}$-sample locations. There exist a number of popular design grids: randomly generated points, Quasi Monte Carlo points, specially designed points,
etc., see e.g., \cite{NonI,NonII,Point,XUZHOU}. It is known that
obtaining the \textit{optimal} sample design is not straightforward
as demonstrated by a recent comparison work in \cite{Comparison}.
Analysis for the least-squares approach utilizing random points is addressed in several contexts, see e.g., \cite{FabioL2,Cohen,XUZHOU}. Generally speaking, the least square approach is stable when the number of sample points behaves quadratically with the dimension of the approximation space. This quadratic condition can be weakened
if we seal with the Chebyshev measure \cite{Cohen2014}.

Note that all the above results are for random parameters
in {\em bounded} domains. As far as we have known,
there is no exhaustive investigations for problems in
unbounded domains, i.e., for functions $f(\mathbf{y})$
with Gaussian or Gamma random parameters.
In this paper, we will consider the problem of approximating
functions with Gaussian or Gamma random parameters by
using discrete least-squares projection
with random points evaluations.
In this case, the traditional approach is to use the so-called
Hermite or Laguerre chaos expansions, where the collocation points
with respect to the Gaussian or Gamma measure will be generated.
However, we will show that such an approach suffers from
the \textit{instability} in the sense that the corresponding
design matrices in the least square approach are well
conditioned \textit{only} when the number of random points
is \textit{exponentially} related to the dimension of
the approximation space, i.e. the number of random points
equals to $(\#\Lambda)^{c\#\Lambda}$ with $\#\Lambda$ being
the dimension of the approximation space. This is obviously
\textit{unacceptable} for practical computations.

To improve the stability we will propose to use
the Hermite (Laguerre) {\em function} approximation
to replace the Hermite (Laguerre) polynomial approach.
Then the mapped uniformly distributed random points are used
to control the condition number of the design matrix.
By choosing a suitable mapping parameter, it is
demonstrated numerically that these two strategies will make the
condition number small
provided that the number of design points is \textit{linearly}
proportional to the dimension of the approximation space.
This stability result is further justified by a theoretical
proof.

The rate of convergence is another serious issue. In fact,
approximating a function by Hermite polynomials or functions
was rejected by Gottlieb-Orszag (\cite{Got77}, pp. 44-45). They
pointed out that
{\em to study the rate of convergence of Hermite series,
we consider the expansion of $\sin(x)$ ...  The result is very bad:
to resolve $M$ wavelengths of $\sin(x)$ requires nearly
$M^2$ Hermite polynomials!  Because of the poor resolution
properties of Hermite polynomials the authors doubt they
will be of much practical value in applications of spectral methods.}

How to improve the resolution property of the Hermite
expansion methods? One remedy is to use the so-called scaling
factor which expands the underlying function by $h_n(\alpha x)$
instead of $h_n(x)$, where $\alpha >0$ is a properly chosen constant. In \cite{Scaling},
a scaling factor formula combining the size of the
solution decay rate and the roots of $h_N(x)$ is proposed,
where $N$ is  the largest expansion term in the Hermite
spectral expansion. Numerical analysis based on asymptotic analysis
numerical experiments demonstrate that the use of the scaling factor
can indeed provide a significant improvement over the observation
on Gottlieb and Orszag.
The theoretical justification of the use of
the scaling factor proposed in \cite{Scaling}
was made in \cite{Guo,Sun}. In particular,
Hermite spectral methods are investigated in \cite{Sun} for linear
diffusion equations and nonlinear convection-diffusion equations in
unbounded domains. When the solution domain is unbounded,
the diffusion operator no longer has a compact resolvent,
which makes the Hermite spectral methods {\em unstable}.
To overcome this difficulty, a time-dependent scaling factor
is employed in the Hermite expansions, which yields a positive
bilinear form. As a consequence, stability is recovered and
spectral convergence speed is significantly enhanced.
In fact, in the past ten years, the use of the scaling factor
proposed in \cite{Scaling} has been used in many areas
including computational optics \cite{xx1},
computational astrophysics \cite{xx2},
etc. In particular, the
scaling factor formula is included in the recent MATLAB code GSGPEs
\cite{CPC13}.

When studying uncertainty using the gPC methods,
Karniadakis, Xiu etc pointed out in \cite{Xiu0,XiuK1} that
the relatively poor resolution properties of Hermite and
Laguerre expansions are well documented in \cite{Got77}.
They further pointed out
the re-scaling procedure as done in \cite{Scaling}
can be employed to accelerate convergence.
However, the progress of using the scaling factor for
the UQ problems has not been big. This is one of the main
motivations of the present work. In this work, we will
introduce suitable scaling factors to speed up the convergence.
Applications to parametric UQ are discussed by considering
random ODE models and elliptic type problems
with lognormal random input. A number of
numerical examples are provided to confirm the efficiency of
the Hermite (Laguerre) function approach with the use of the
scaling factors.

We summarize here the distinct features of our approach:
\begin{itemize}
\item
We investigate the discrete least square approach for
functions with Gaussian or Gamma random parameters;
applications to UQ are discussed.
\item
We propose to use the Hermite (Laguerre) functions
as the approximation bases, which is different with
the traditional Hermite (Laguerre) polynomials.
Stability is guaranteed with acceptable number of evaluation points
and relevant theoretical justification is provided.
\item
We introduce the scaling factor in the least square
approach to speed up the convergence, and a principle for
choosing the scaling is provided. The numerical results
indicate that the use of the proposed scaling factor
is indeed very useful.
\end{itemize}

The rest of this paper is organized as follows. In section 2,
we introduce the approximation problem of a function
in $d$-dimensions by discrete least-square projection.
Some commonly used high dimensional approximation spaces
are discussed. We also show that the Hermite (Laguerre) gPC
expansions need an unacceptable number of evaluation points
to guarantee the stability. In Section 3, we propose to use
the Hermite (Laguerre) function approach.
Stability under this approach is ensured with the use
of the mapped uniform random points. Moreover,
a useful scaling factor is introduced to speed up the convergence.
Applications to parametric UQ are discussed in Section 4.
Some conclusions will be drawn in the final section.

\section{The least square projection}
\setcounter{equation}{0}

In this section, we follow closely the works
\cite{FabioL2,Cohen,XUZHOU} to give a basic introduction for the discrete least-squares approach, however, please note that we shall focus on problems in unbounded domains.

Let $\mathbf{y}=(y_1, \cdot\cdot\cdot, y_d)^T$ be a vector with $d$ random variables, which takes values in  $\Gamma\equiv \mathbb{R}^d$ or $\Gamma\equiv \mathbb{R}_+^d.$ We will focus on the cases where $\{y_i\}_{i=1}^d$ are  Gaussian random variables ($\Gamma\equiv \mathbb{R}^d$) or  Gamma random variables ($\Gamma\equiv \mathbb{R}_+^d$).  We suppose that the variables $\{y_i\}_{i=1}^d$ are independent with marginal probability density function (PDF) $\rho_i$ for each random variable $y_i.$  The joint PDF is given by $\rho(\mathbf{y})= \prod_{i=1}^d \rho_i(y_i): \Gamma \rightarrow \mathbb{R}^+$.

Assume that the functions considered in this paper are in the space $L_\rho^2$ endowed with the norm
\begin{equation}
||f||_{L^2_\rho}= \E \left[f^2(\mathbf{y})\right] = \Big(\int_\Gamma f^2(\mathbf{y}) \rho(\mathbf{y}) d\mathbf{y}\Big)^{1/2}.
\end{equation}
The purpose is to efficiently build a finite dimension approximation of $f(\mathbf{y})$ or some general functionals $g \!\circ\! f$ associated with $f(\mathbf{y}).$  To this end, we first choose the
one-dimensional orthogonal bases (not only limited to polynomials)
with respect to each random variable $y^i$:
$$\{\phi^i_j\}_{j=1}^{\infty} \in L_\rho^2, \quad  i=1,...,d,$$
where  $\phi^i_j$ is called the $j$-th order basis. Then the multi-dimensional bases can be formed by tensorizing the univariate bases $\{\phi^i_j\}_{j=1}^{\infty}.$
To explicitly form these bases, let us first define the following multi-index:
\begin{align*}
 \mathbf{n}=(n_1,\cdot\cdot\cdot,n_d)\in \mathds{N}^d, \quad \textrm{with} \quad  |\mathbf{n}|=\sum_{i=1}^d n_i.
\end{align*}
Define the $d$-dimensional bases $\mathbf{\Phi}_\mathbf{n}$ as
\begin{equation}
\mathbf{\Phi}_\mathbf{n}(\mathbf{y})=\prod_{i=1}^d \phi^i_{n_i}(y_i),
\end{equation}
where $\{\phi^i_{n_i}\}_{n_i=1}^\infty$ is the one-dimension basis.
Let $\Lambda\subset \mathds{N}^d$ be a finite multi-index set, and denote by $N:=\# \Lambda$ the cardinality of an index set $\Lambda.$ The finite dimensional approximation space defined by $\Lambda$ is given by
\begin{align*}
\mathbf{P}^\Lambda :=\textmd{span}\{\mathbf{\Phi}_\mathbf{n}(\mathbf{y}), \,\,\mathbf{n}\in \Lambda \}.
\end{align*}
Throughout the paper, the best approximation of $f(\mathbf{y})$ in $\mathbf{P}^\Lambda$ will be denoted by $P^\Lambda f,$ namely,
\begin{align}\label{eq:best}
P^\Lambda f := \argmin_{p\in \mathbf{P}^\Lambda} \|f-p\|_{L^2_\rho}.
\end{align}
A formula for the best approximation $P^\Lambda f$ involves standard Fourier coefficients with respect to the $\Phi_n$, but these coefficients require high-order moment information for the function $f$ and in general cannot be computed explicitly.

Alternatively, we consider the construction of such an
approximation $f^\Lambda \in \mathbf{P}^\Lambda$ for the function $Z=f(\mathbf{y})$ by the least-squares approach. To this end,
we compute the exact function values of $f$ at $\mathbf{y}_1,
\ldots , \mathbf{y}_m\in \R^d$ with  $m> N$, and then find
a discrete least-squares approximation $f^\Lambda$ by requiring
\begin{align}\label{eq:least}
f^\Lambda= P^\Lambda_m f = \argmin_{p\in \mathbf{P}^\Lambda}  \sum_{k=1}^m\left(p(\mathbf{y}_k)-f(\mathbf{y}_k)\right)^2.
\end{align}
We introduce the discrete inner product
\begin{align}\label{eq:least_norm}
\innerp{u, v}_m= \sum_{k=1}^m u(\mathbf{y}_k)v(\mathbf{y}_k).
\end{align}

\begin{remark}\label{rm:non_poly}
{\rm
We remark that usually the $L^2_\rho$-best approximation
polynomial is chosen as the approximation bases,
which yields the so-called gPC  method. For example,
the Hermite polynomials are used for functions with Gaussian
parameters, and the Leguerre polynomials are suitable
for functions with Gamma parameters, and so on \cite{XiuK2}.
In such gPC expansions, a natural way to choose the design points
is the random sampling method, that is, the random samples are
generated with respect to $\mathbf{\rho}.$
Of course, other kinds (non-polynomials) of
orthogonal bases can be used in the least-square approach.
}
\end{remark}

\subsection{Multivariate approximation spaces}

Given a basis order $q$ and the dimension parameter $d\in \N,$
define the following index sets
\begin{eqnarray}
&& \Lambda_{\bf P}^{q,d}:=\{{\bf n}=(n_1,\ldots,n_d)\in \N^d:\max_{j=1,\ldots,d}n_j\leq q \},
\\
&& \Lambda_{\bf D}^{q,d}:=\{{\bf n}=(n_1,\ldots,n_d)\in \N^d:\abs{{\bf n}}\leq q \}.
\end{eqnarray}
The traditional tensor product (TP) space is defined as
\begin{equation}
\mathbf{P}_q^d \,\,:=\,\, {\rm span} \big\{\mathbf{\Phi}_\mathbf{n}(\mathbf{y}):  \mathbf{n}\in \Lambda_{\bf P}^{q,d} \big\}.
\end{equation}
That is, we require in $\mathbf{P}_q^d$ that the basis order in each
variable less than or equal to $q$. A simple observation is that the dimension of $\mathbf{P}_q^d$ is
\begin{equation}
{\rm dim}(\mathbf{P}^d_q)=\#\Lambda_{\bf P}^{q,d}=(q+1)^d.
\end{equation}
Note that when $d\gg 1$ the dimension of TP spaces grows very quickly with respect to the degree $q$, which  is the so-called \textit{curse of dimensionality}. As a result, the TP spaces are rarely used in practice when $d$ is large. Alternatively,
when $d$ is large, the following total degree (TD) space is often
employed instead of using the TP space \cite{FabioC2,XiuH}:
\begin{equation}
\mathbf{D}^d_q \,\,:=\,\, {\rm span} \big\{\mathbf{\Phi}_\mathbf{n}(\mathbf{x}):  \mathbf{n}\in \Lambda_{\bf D}^{q,d} \big\}.
\end{equation}
The dimension of $\mathbf{D}^d_q $ is
\begin{equation}
{\rm dim}(\mathbf{D}^d_q)=\#\Lambda_{\bf D}^{q,d}={q+d\choose d}.
\end{equation}
It is seen that the growth of the dimension of $\mathbf{D}_q^d$ is much slower than that of $\mathbf{P}_q^d$.

\begin{remark}\label{rm:Polynomial}
{\rm
We remark that the TP and TD spaces are originally
defined for polynomial spaces. However, spaces based on
general one-dimensional bases can be constructed using the same
way. Consequently, we will still use the names of
TP and TD for the spaces with general bases. Moreover,
other types of multi-variate approximation spaces
can be constructed in a similar way, e.g.,
the hyperbolic cross \cite{Cohen2014}.
}
\end{remark}

\subsection{Algebraic formulation}
Consider the approximation in the space $\mathbf{P}^\Lambda=\textmd{span}\{\mathbf{\Phi}_{\bf n}\}_{{\bf n}\in \Lambda}$ with random samples $\{\mathbf{y}_k\}_{k=1}^{m}.$
If we choose a proper ordering scheme for the multi-index, we can order the multi-dimensional bases via
a single index.
 For example, we can  arrange the index set $\Lambda$ in the lexicographical order, namely, given $\mathbf{n}^\prime, \mathbf{n}^{\prime\prime} \in \Lambda$
\begin{equation}
\mathbf{n}^\prime < \mathbf{n}^{\prime\prime} \Leftrightarrow
 \big[\; |\mathbf{n^\prime}| < |\mathbf{n^{\prime\prime}}|\;
\big] \vee
\big[ \left(\; |\mathbf{n^\prime}| = |\mathbf{n^{\prime\prime}}|\;\right) \wedge
\left( \exists \,j \,:\, n^\prime_j < n^{\prime\prime}_j \wedge (n^\prime_i = n^{\prime\prime}_i, \,\, \forall i< j) \right) \big].
\end{equation}
Then the space $\mathbf{P}^\Lambda$ can be rewritten as
 $\mathbf{P}^\Lambda=\textmd{span}\{\{\mathbf{\Phi}_{\bf n}\}_{j=1}^N\}$ with $N=\#\Lambda$.
The least square solution can be written in
\begin{equation}
f^\Lambda=\sum_{j=1}^{N} c_j \mathbf{\Phi}_j,
\end{equation}
where $\mathbf{c}= (c_1,...,c_{N})^\top$ is the coefficient vector.
The algebraic problem to determine the unknown coefficient $\mathbf{c}$ can be
formulated as:
\begin{align}\label{eq:le}
\mathbf{c}=\argmin_{\mathbf{z}\in \mathbb{R}^{N}} ||\mathbf{D}\mathbf{z}-\mathbf{b}||_2,
\end{align}
where
\begin{align*}
\mathbf{D}=\Big( \mathbf{\Phi}_j(\mathbf{y}_k)\Big),
\quad j=1,...,N,\,\, k=1,...,m,
\end{align*}
and
$\mathbf{b}=[f({\mathbf y}_1),\ldots,f({\mathbf y}_m)]^\top$ contains the evaluations of the target function $f$ in the collocation points.
The solution to the least squares problem  (\ref{eq:le}) can also be computed by
solving an $N \times N$ system, namely,
\begin{align}\label{eq:solution}
\mathbf{A} \mathbf{z} &= \mathbf{f}
\end{align}
with
\begin{align}\label{eq:components}
\mathbf{A}:=\mathbf{D}^\top\mathbf{D}= \Big(\innerp{ \mathbf{\Phi}_i, \mathbf{\Phi}_j}_m \Big)_{i,j=1,...,N}, \quad \mathbf{f}:=\mathbf{D}^\top\mathbf{b}=\Big(\innerp{  f, \mathbf{\Phi}_j}_m \Big)_{j=1,...,N}.
\end{align}
For the computation point of view, we can solve problem (\ref{eq:le}) by using the QR factorization. Alternatively, we can also solve (\ref{eq:solution}) by the Cholesky factorization.

\subsection{The Hermite (Laguerre) chaos expansion: stability issue}
As was discussed in Remark \ref{rm:non_poly}, a nature way to approximate functions with Gaussian (Gamma) parameters is the Hermite (Laguerre) chaos expansion. In this section, we shall show, by numerical examples, that the least square projection with Hermite (Laguerre) polynomials expansion is unstable, in the sense that an unfeasible number of random points, i.e.,
\[
m=(\#\Lambda)^{c\#\Lambda},\]
are needed to guarantee the stability.

To this end, let us remind that the one-dimensional normalized Hermite polynomials $\{H_k(y)\}_{k=0}^\infty,$ defined on the whole line $\mathbb{R}:=(-\infty,+\infty),$ are orthogonal with respect to the weight function
\begin{equation}
\rho^G(y)=\textmd{e}^{-y^2},
\end{equation}
 namely,
\begin{equation}
\int_{-\infty}^{+\infty} \rho^G(y) H_m(y)H_n(y) dy = \delta_{mn}.
\end{equation}
We denote by $\mathbf{H}_\mathbf{n}(\mathbf{y})$ the multi-variate hermite polynomial with multi-index $\mathbf{n},$ which obtained by tensorized the one-dimensional Hermite polynomials. Then, a natural way to approximate a multivariate function $f^G(\mathbf{y})$ with Gaussian parameters $\mathbf{y}$ is
\begin{equation}\label{eq:hermite}
f^G(\mathbf{y})=\sum_{n} c_\mathbf{n} \mathbf{H}_\mathbf{n}(\mathbf{y}), \quad \mathbf{n} \in \Lambda,
\end{equation}
where $\Lambda $ is the index set that can be either $\Lambda_{\bf P}^{q,d}$ or $\Lambda_{\bf D}^{q,d}.$

Similarly, for a function $f^E(\mathbf{y})$ with Gamma random
parameters $\mathbf{y},$  a nature bases for such an expansion
would be the tensorized Laguerre polynomials
$\mathbf{L}_\mathbf{n}$ that are orthogonal with respect to the
weight function $\rho^E(\mathbf{y})=\prod_{i=1}^d
\textmd{e}^{-y_i}$. More precisely, we expand
\begin{equation}\label{eq:Laguerre}
f^E(\mathbf{Y})=\sum_{n} c_\mathbf{n} \mathbf{L}_\mathbf{n}
(\mathbf{y}), \quad \mathbf{n} \in \Lambda.
\end{equation}
Note that we consider here a special type of Gamma random parameters $y,$ for which the PDF yields $\rho(y)=\textmd{e}^{-y}.$ Such random variables are also referred to the exponential random variables.
More general types of Gamma random parameters with PDF
\begin{equation}
\rho^{E}(y)=\frac{\beta^\alpha y^{\alpha-1}\textmd{e}^{-\beta y}}
{\Gamma(\alpha)}
\end{equation}
 can be considered in a similar way, and the corresponding chaos
expansion is the generalized Laguerre chaos expansions.

In the least square framework, to construct the expansions (\ref{eq:hermite}) and (\ref{eq:Laguerre}), a natural choice of the collocation points $\{\mathbf{y}\}_{i=1}^m$ is to generate random points according to the Gaussian (Gamma) measure.  In both cases, we can obtain the corresponding design matrices $\mathbf{A}^G$ and $\mathbf{A}^E$, respectively.

\begin{figure}[t]
\begin{center}
\includegraphics[width=0.43\textwidth]{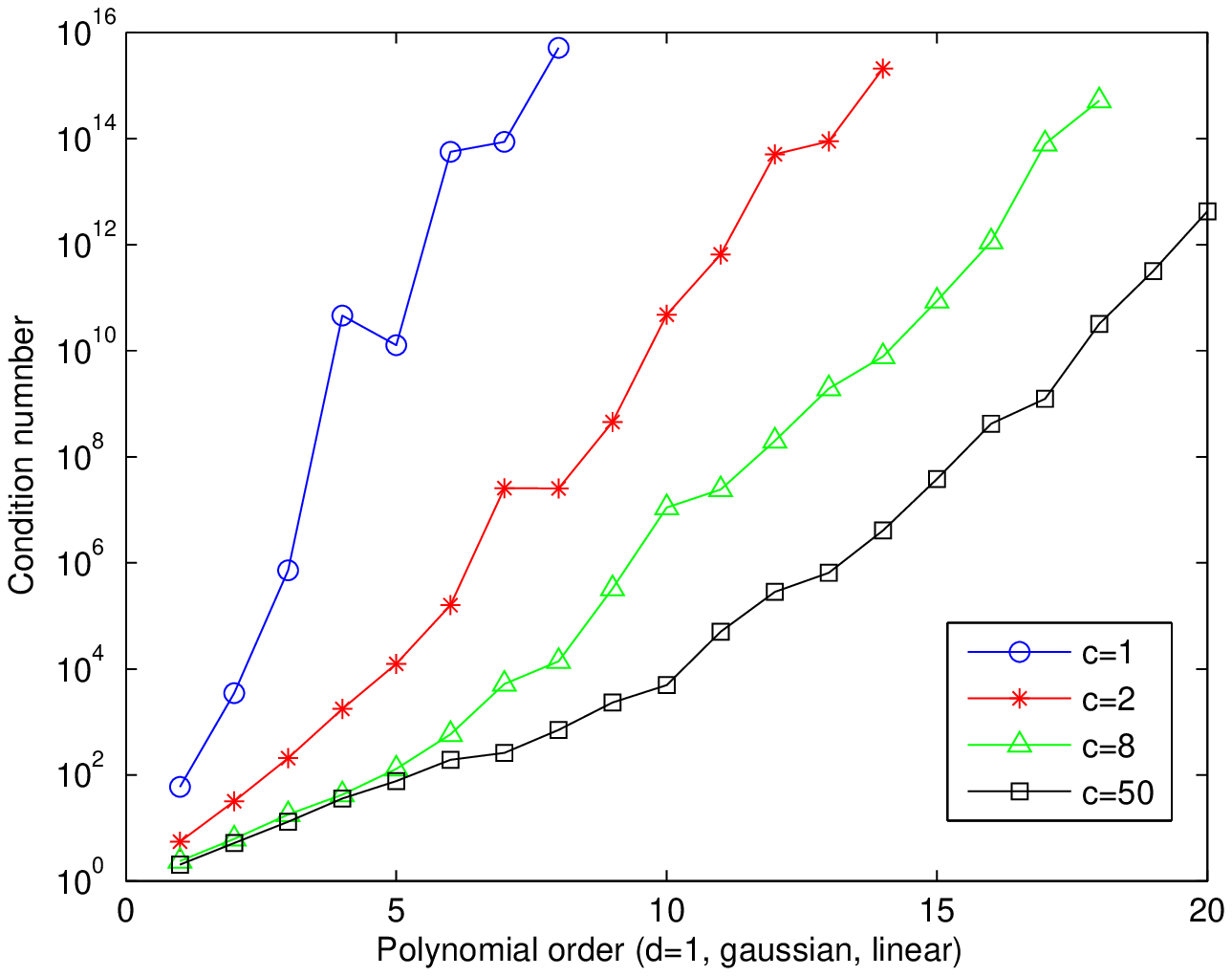}
\includegraphics[width=0.43\textwidth]{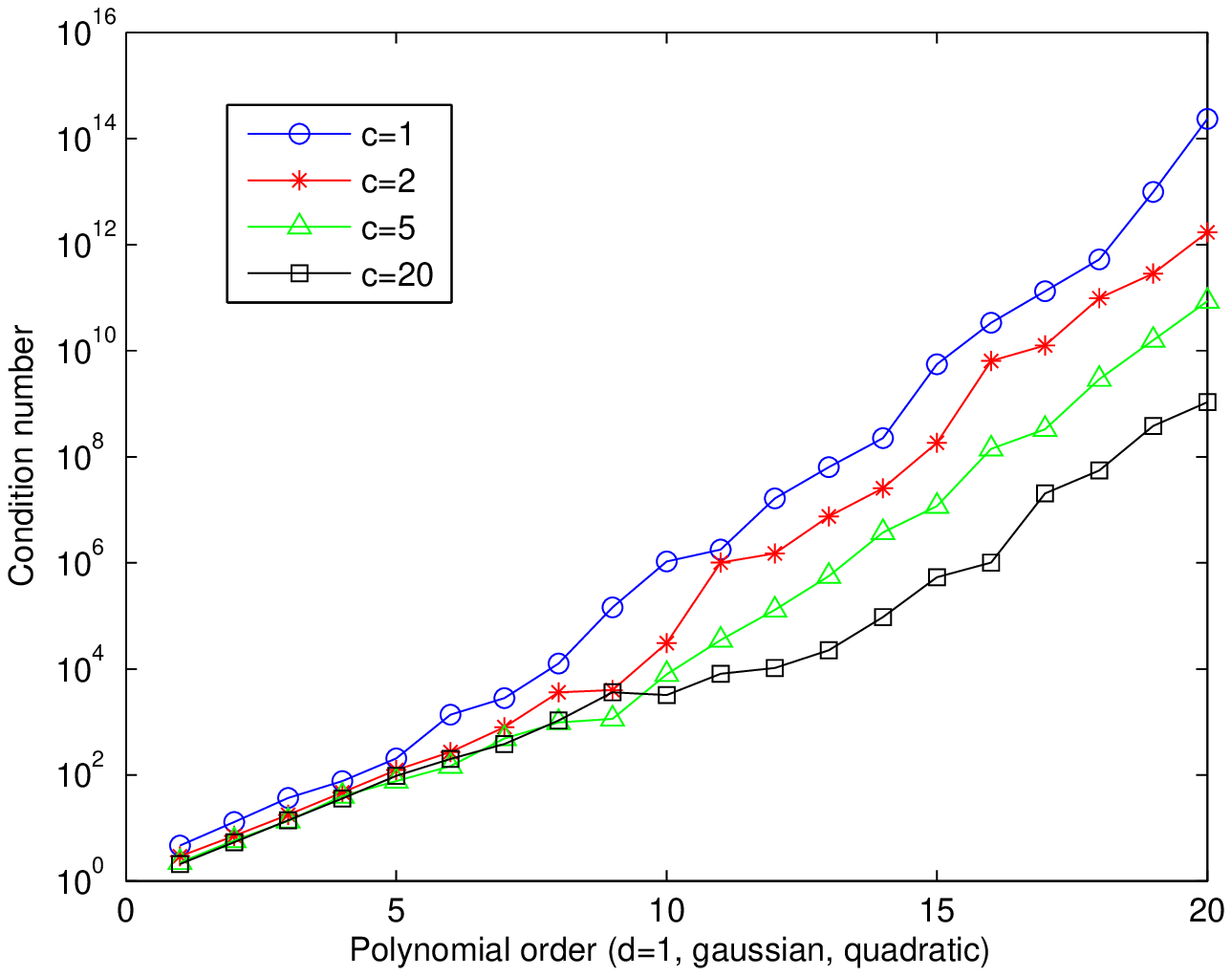}
\end{center}
\begin{center}
\includegraphics[width=0.43\textwidth]{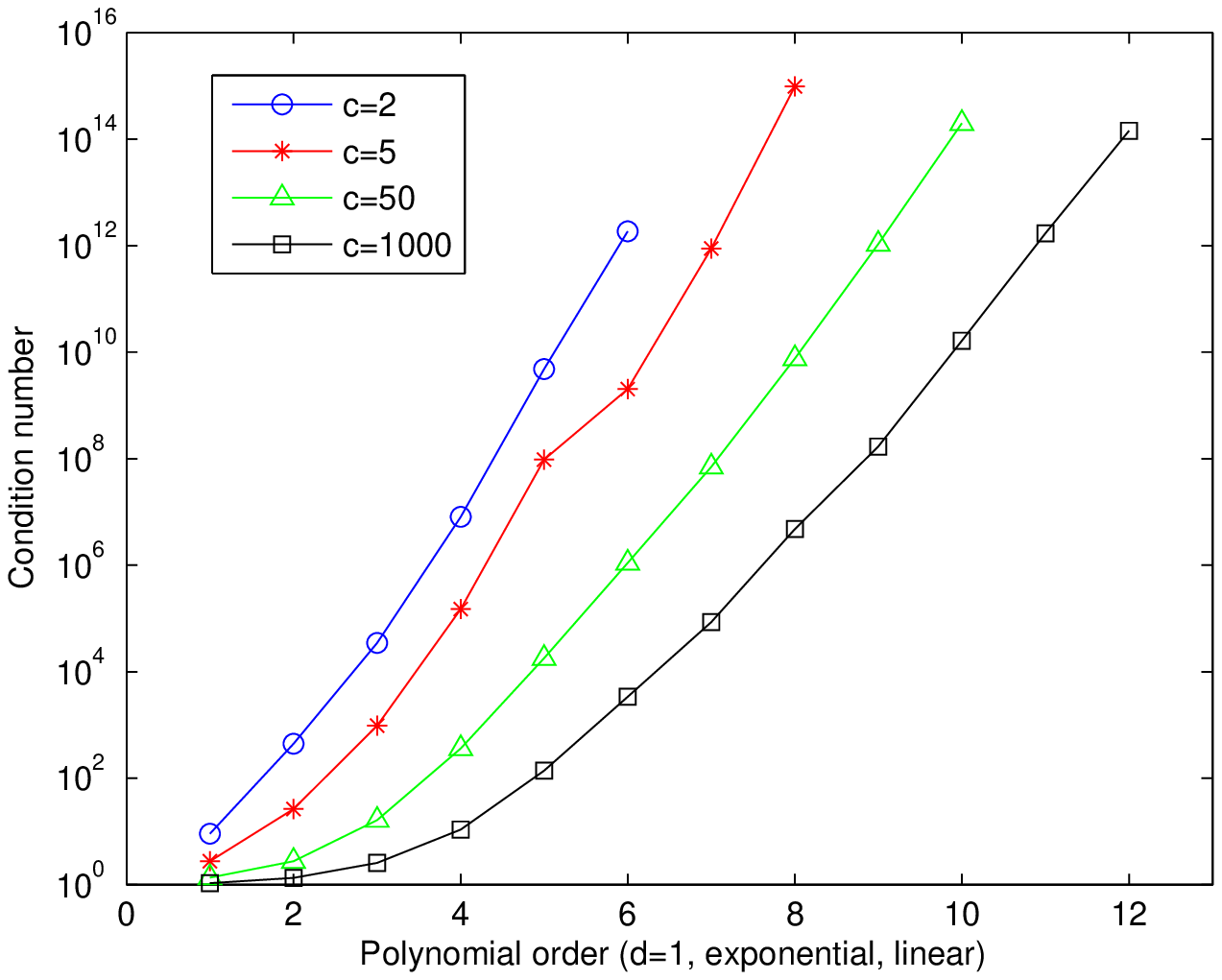}
\includegraphics[width=0.43\textwidth]{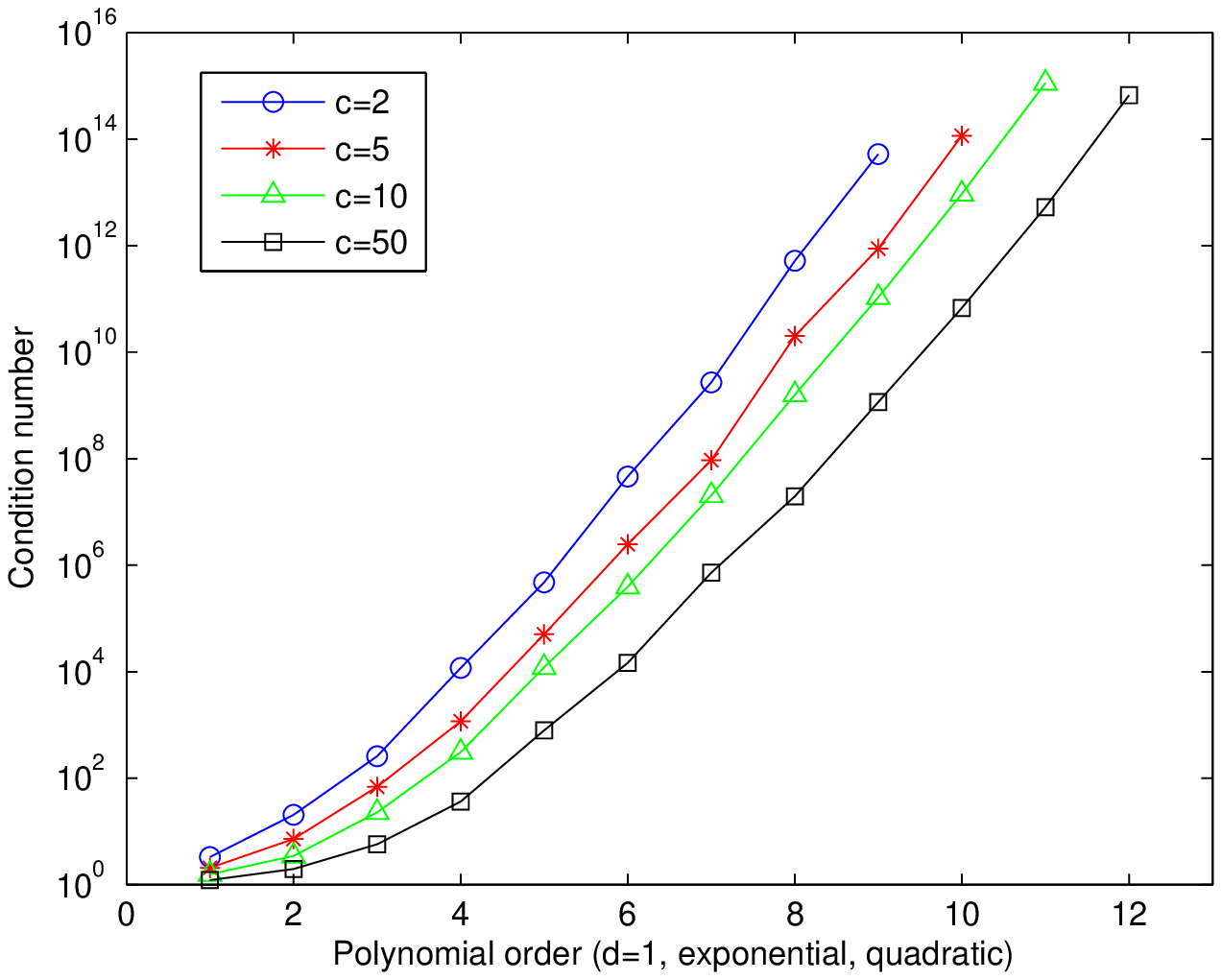}
\end{center}
\caption{\small
Condition numbers with respect to polynomial order in
1D case, with left for $m=c(\#\Lambda)$ and  right
for $m=c(\#\Lambda)^2.$
Top: Gaussian case; Bottom: Gamma case.}
\label{fig1_condition}
\end{figure}

We remark that for problems in \textit{bounded} domains,
e.g., the uniform random parameters in $[-1,1]$,
the relevant tests have been done by many researchers,
see, e.g.,  \cite{FabioL2,Cohen,Cohen2014,XUZHOU}.
For instance for the uniform measure in $[-1,1],$ it is  known that a quadratic
dependence of the number of random points, i.e.
$m=c (\# \Lambda)^2$, is sufficient to guarantee
the stability of the least square approach. Moreover, if
the Chebyshev measure is considered, fewer points are needed
to guarantee the stability \cite{Cohen2014}.

What is the difference if the underlying domain is
unbounded? The answer is quite negative: the  $m=c (\# \Lambda)^2$
quadratic random points cannot guarantee the stability.

We will demonstrate the above claim by testing the condition number of the design matrices, i.e.,
\begin{equation}
\textmd{cond}(\mathbf{A})=\frac{\sigma_{max}(\mathbf{A})}{\sigma_{min}(\mathbf{A})},  \quad \mathbf{A}=\mathbf{A}^G \,\,\textmd{or}\,\,\mathbf{A}^E.
\end{equation}
Let us first consider the Hermite chaos expansion
(\ref{eq:hermite}). In this case, the random points
are generated with respect to the Gaussian measure.
Note that the design matrix is a random matrix. Therefore,
in the computations we will repeat the test for 100 times,
and the \textit{mean} condition number will be reported.
In Fig. \ref{fig1_condition}, the growth of
condition numbers with respect to the polynomial order
is shown for the one-dimensional case. It is noted that
the condition number admits an exponential growth with
respect to the polynomial order, for both the linear
dependence $m=c (\# \Lambda) $ (left) and the quadratic
dependence $m=c(\# \Lambda)^2$ (right)
cases. In fact, similar tests with the dependence
$m=c (\# \Lambda)^\nu $ with $3 \leq \nu \leq 5$
produce results similar to those in Fig.
\ref{fig1_condition}.

\begin{figure}[h]
\begin{center}
\includegraphics[width=0.43\textwidth]{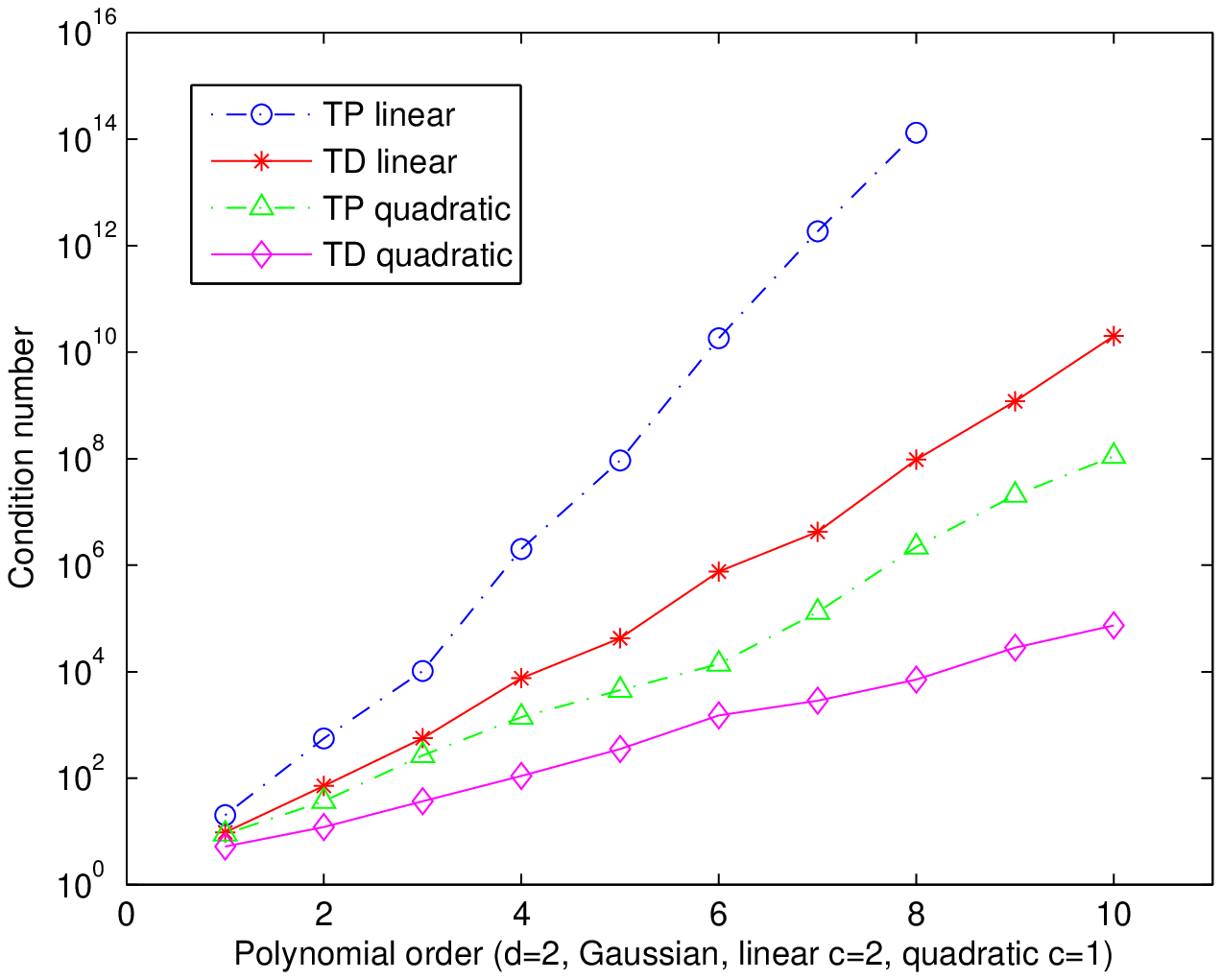}
\includegraphics[width=0.43\textwidth]{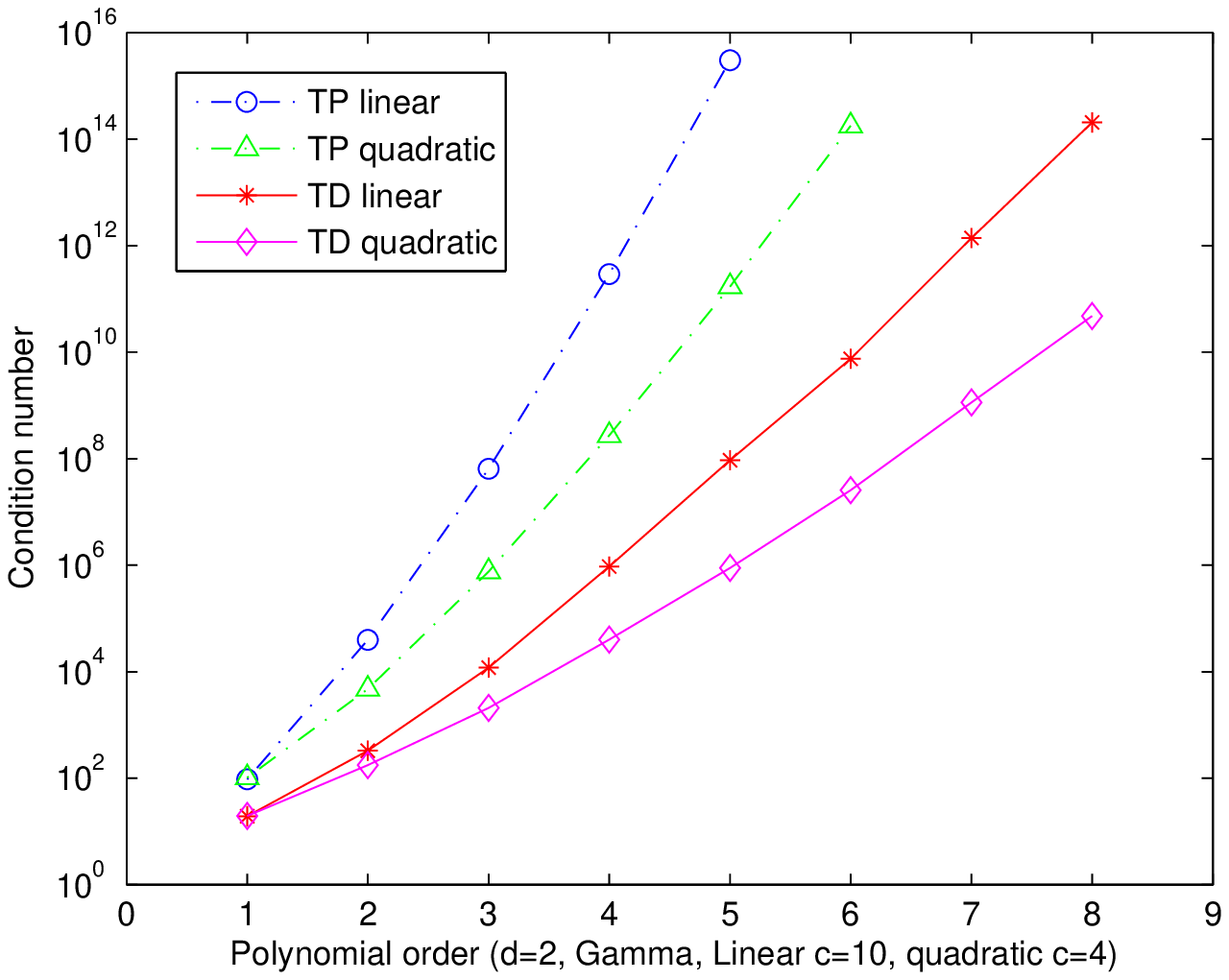}
\end{center}
\caption{Condition numbers with respect to polynomial order in the 2 dimensional case. Left: Gaussian. Right: Gamma.}\label{fig:gaussian_gamma_2_condition}
\end{figure}

We further consider the Laguerre Chaos expansion, which is
suitable for approximating functions supported in
$\mathbb{R}_+^d.$ Note that the corresponding random points
are generated by the Gamma measure.  The bottom of
Fig. \ref{fig1_condition} shows the results for one-dimensional
tests, which indicate that the condition number of the
Gamma case grows faster than that in Gaussian.

Fig. \ref{fig:gaussian_gamma_2_condition} presents
the two-dimensional tests for both the TP and TD constructions.
The left figure is for the Gaussian,
while the right one is for the Gamma.  Again
the exponential growth of the condition number
is observed again, where it is seen that the TD spaces
work better than the TP spaces.

With the above observations, it seems hopeless to control the
condition number in the unbounded domain. In fact,
to have a good control of the condition number,
it is observed in the thesis of G. Migliorati \cite{Thesis} an unfeasible number of points with
$m= (\# \Lambda)^{c (\# \Lambda)}$ is needed.
To improve this, we shall introduce the Hermite (Laguerre)
function approach to replace the Hermite (Laguerre) polynomial
expansion.

\begin{remark}
{\rm
We remark that we are not saying that the Hermite (Laguerre)
polynomial chaos expansions are unfeasible in the least
square framework. In fact, we can still use such approaches
with {\em small} number of  polynomial degrees.
In this case, fast convergence can still be expected.
However, the convergence rate deteriorates when a large
number of polynomial degree $q$ is used due to the exponential
growth of the condition number. Some
numerical tests are provided in \cite{Thesis}.
}
\end{remark}

\section{The Hermite (Laguerre) function expansions}
\setcounter{equation}{0}

In this section, we propose to use the Hermite (Laguerre) function approximation instead of the traditional Hermite (Laguerre) polynomial approximation. The one-dimensional Hermite functions, also named modified Hermite polynomials, are defined by
\begin{equation}
\label{hf}
\tilde{H}_m(y)=\textmd{e}^{-\frac{y^2}{2}} H_m(y), \quad m=0,1,...
\end{equation}
where $\{H_m(y)\}_{m\geq 0}$ are normalized Hermite polynomials. Note that the Hermite functions are orthogonal in the following sense
\begin{equation}
\int_{-\infty}^{+\infty}  \tilde{H}_m(y)\tilde{H}_n(y) dy = \delta_{mn}.
\end{equation}
The corresponding multivariate Hermite functions $\mathbf{\tilde{H}}_\mathbf{m}(\mathbf{y})$ can be defined by tensorizing the one dimensional Hermite functions.

The Laguerre functions are defined as
\begin{equation}
\tilde{L}_m(y)=\textmd{e}^{-\frac{y}{2}} L_m(y), \quad m=0,1,...,
\end{equation}
where $\{L_m(y)\}_{m\geq 0}$ are Laguerre polynomials.  The corresponding multi-variate Laguerre functions $\mathbf{\tilde{L}}_\mathbf{m}(\mathbf{y})$ can be defined in a similar way. Note that the Hermite/Laguerre functions are no longer polynomials. Nevertheless,
in what follows, whenever we use \textit{polynomial order} $q$
it is referring to the \textit{q}th Hermite/Laguerre function.

It is clear that the Hermite (Laguerre) function expansions are suitable for approximating functions decaying  to zero when $y$ goes to infinity. We claim that in UQ applications, we can almost always consider approximating decay functions. To see this, let $f(y)$ (scalar case, for simplicity) be a function with Gaussian parameters, that might be the solution of certain stochastic ODEs/PDEs. In the UQ applications, one is interested in some statistic quantities of $f(y),$ such as the $k$th moment $\int_\Gamma \rho(y) f^k(y) dy.$ Let us consider a general expression of such QoI:
\begin{equation}\label{eq:QoI}
\textmd{QoI} = \int_\Gamma \rho(y) (g\!\circ\!f) (y) dy.
\end{equation}
where $g\!\circ\!f$ is a general smooth functional of $f(y).$  Even if $g\!\circ\!f$ is not a decay function, $\rho(y) (g\!\circ\!f)$ does, provided that $g\!\circ\! f$ grows slower than Gaussian. Thus, we can in fact consider the approximation for $\tilde{f}(y)= \rho(y) (g\!\circ \!f).$  As long as a good approximation of $\tilde{f}(y)$ is found, we can get a good approximation for the QoI in (\ref{eq:QoI}).

Without loss of generality, we can assume that $f(y)$ decays exponentially. Consider the expansion
\begin{equation}\label{eq:herlagfun}
f^G(y)=\sum_{n=0}^{K-1} c_n \tilde{H}_n(y), \quad
f^E(y)=\sum_{n=0}^{K-1} c_n \tilde{L}_n(y).
\end{equation}
We are now at the stage to find good collocation points in the least square framework. As we have discussed before, the most natural way to find such points is to generate the samples with respect to the PDF of the random parameters. Moreover, if such a PDF coincides with the weight function of the bases, the expectation of the design matrix would be the identical matrix, and this feature would help for the rigorous stability analysis \cite{Cohen}. In our setting, however, the Hermite (Laguerre) functions are orthogonal with respect to the Lebesgue measure. It is known that it is {\em impossible} to generate
random points with respect to the Lebesgue measure
(uniform measure) in unbounded domains.
To overcome this difficulty, we shall introduce the mapped
uniform samples, which transform the uniform random points $\{\xi_i\}_{i=1}^m$ in $[-1,1]^d$ (or $[0,1]^d$) to $\{y_i\}_{i=1}^m$ in $[-\infty, +\infty]^d$ (or $[0, +\infty]^d$).

Although there exist many feasible mappings, we shall restrict ourselves to a family of mappings defined by
\begin{equation}\label{eq:mapping}
y'(\xi) =\frac{L}{(1-\xi^2)^{1+r/2}}, \quad r\geq 0,
\end{equation}
where $L > 0$ is a constant, and $r$ determines how fast the mapping $y(\xi)$ goes to
infinity as $\xi$ goes to $\pm 1$, see, e.g.,
\cite{Boyd,SY} for a thorough discussion on
the pros and cons of different mappings. It is easy to verify that
\begin{equation} \label{star1}
y(\xi)=\left\{ \begin{array}{ll}
\frac{L}{2}\log\frac{1+\xi}{1-\xi} &\,\,   r=0,\\
\frac{L\xi}{\sqrt{1-\xi^2}}        &\,\,  r=1,
\end{array} \right. \quad
\xi(y)=\left\{ \begin{array}{ll}
\tanh\left(\frac{y}{L}\right) &\,\,   r=0,\\
\frac{y/L}{\sqrt{y^2/L^2+1}}        &\,\,  r=1.
\end{array} \right.
\end{equation}
For other positive integers $r,$ we can always use an algebraic computing software to derive
the explicit expression of the mapping $y(\xi).$ The mapping with $r=0$ is often referred to the logarithmic mapping
which makes the transformed points decay exponentially, and the mapping with $r > 0$ is referred as algebraic mapping.
In our setting, the mapping with $r=0$ will be used when
the Gaussian measure is considered,
while the mapping with $r=1$ will be adopted when the
Gamma measure is used.

We now summarize our least square approach by taking a
one-dimensional function with Gaussian parameters
as an example. Given the function $f(y)$ to be approximated,
i.e., we are interested in the QoI of
$\int_{\mathbb{R}} \exp(-{y^2}/{2})f(y)dy$.
\begin{itemize}
\item
{\em Step 1.}
 Motivated by the discussion in the beginning of this section,
we seek the following Hermite {\em function} expansion for
$\tilde{f}(y)=\exp(-{y^2}/{2})f(y)$:
\begin{equation}
\tilde{f}(y)=\sum_{k=0}^{K-1} c_i \tilde{H}_k(y).
\end{equation}
\item
{\em Step 2.}
Let $\mathbf{P}^K:=\textmd{span}\{\tilde{H}^0,...,\tilde{H}^{k-1}\}$.
We will find the following least square solution
\begin{align}\label{eq:least_herfun}
f^K= P^K_m f = \argmin_{p\in \mathbf{P}^K}  \sum_{k=1}^m\left(p(y_k)-\tilde{f}(y_k)\right)^2,
\end{align}
where the collocation points $\{y_k\}_{k=1}^m$ are
chosen as the transformed uniform random points given by
the mapping (\ref{eq:mapping}) with $r=0$.
\end{itemize}
This procedure will lead to the desired QoI.

\subsection{Stability}

In this section, we shall investigate the stability of the least square approach by using the Hermite (Laguerre) functions, with mapped uniform distributed random points. Again, we test the condition number of the corresponding design matrices:
\begin{equation}
\textmd{cond}(\mathbf{A})=\frac{\sigma_{max}(\mathbf{A})}{\sigma_{min}(\mathbf{A})},  \quad \mathbf{A}=\mathbf{A}^G \,\,\textmd{or}\,\,\mathbf{A}^E.
\end{equation}
Here we still use $\mathbf{A}$ to avoid too many symbols
although we should point out  that $\mathbf{A}^G$ ($\mathbf{A}^E$)
are evaluations of the Hermite (Laguerre) functions on the
mapped random points in $\mathbb{R}^d$ and $\mathbb{R}_+^d$,
respectively. As such matrices are random,
their condition numbers will be obtained by repeating the
test 100 times so that the resulting mean condition number can be
obtained. The mean condition number will be used to represent
the condition number of the random matrices, which will be reported
in the following figures.

\begin{figure}[t]
\begin{center}
\includegraphics[width=0.43\textwidth]{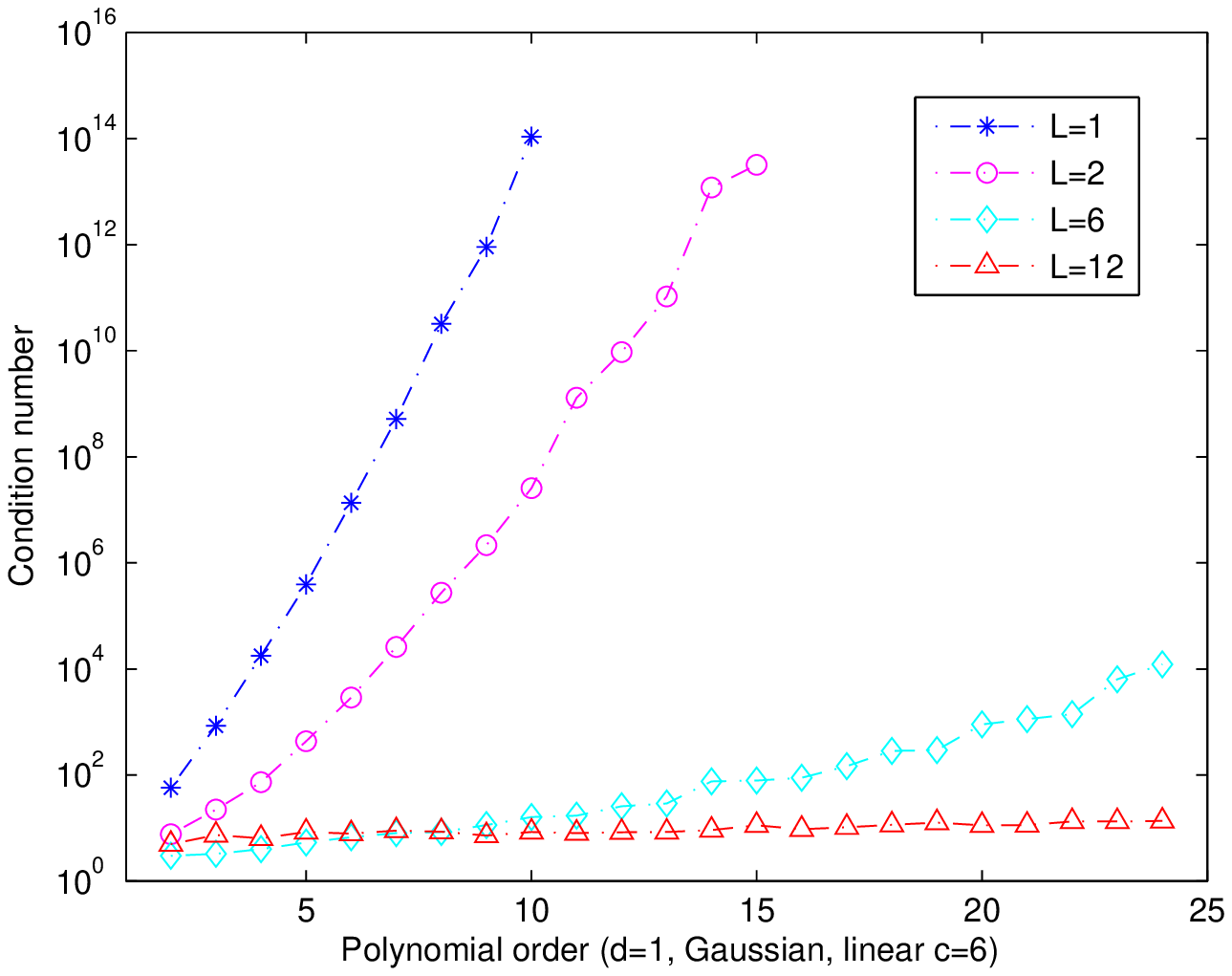}
\includegraphics[width=0.43\textwidth]{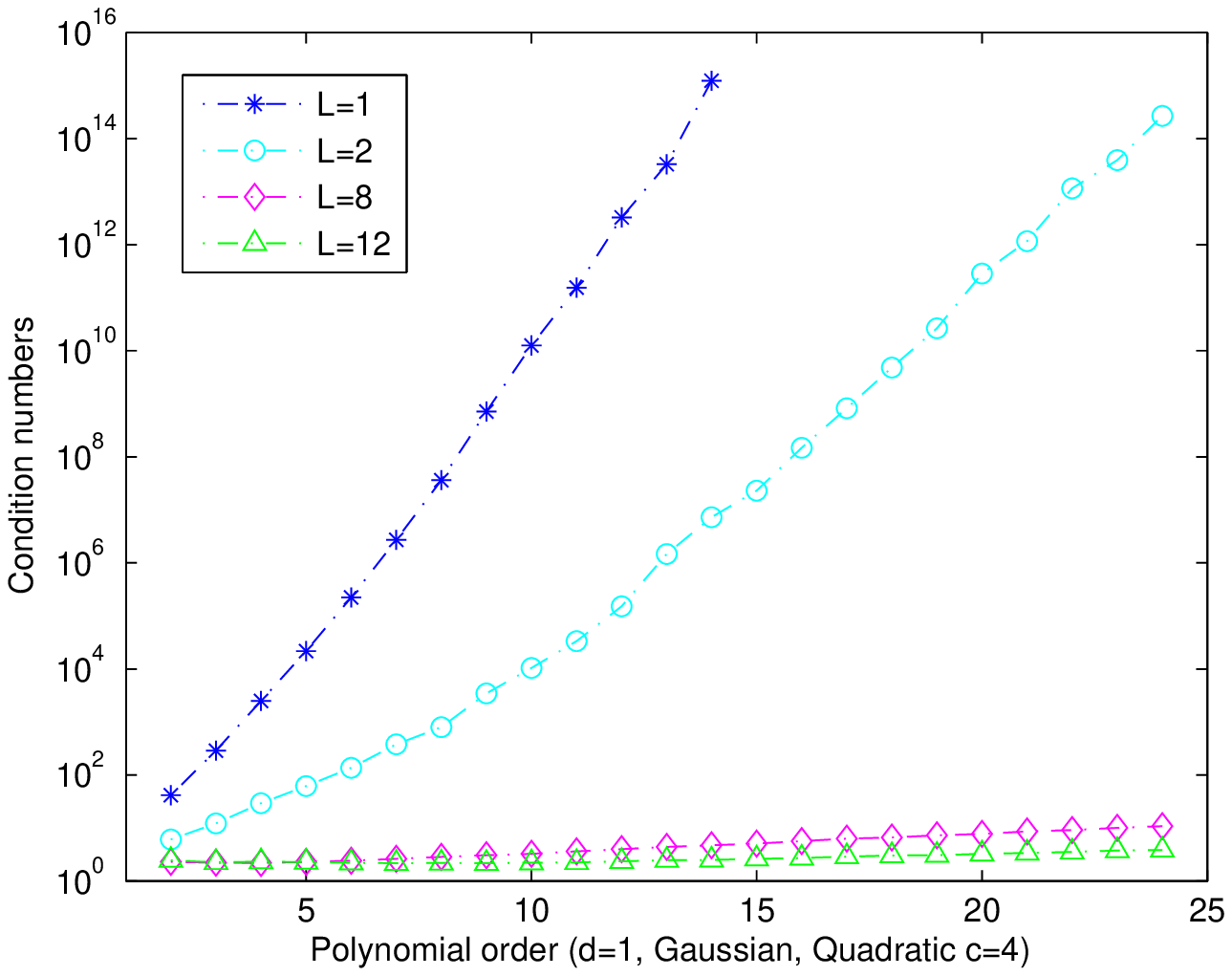}
\end{center}
\begin{center}
\includegraphics[width=0.43\textwidth]{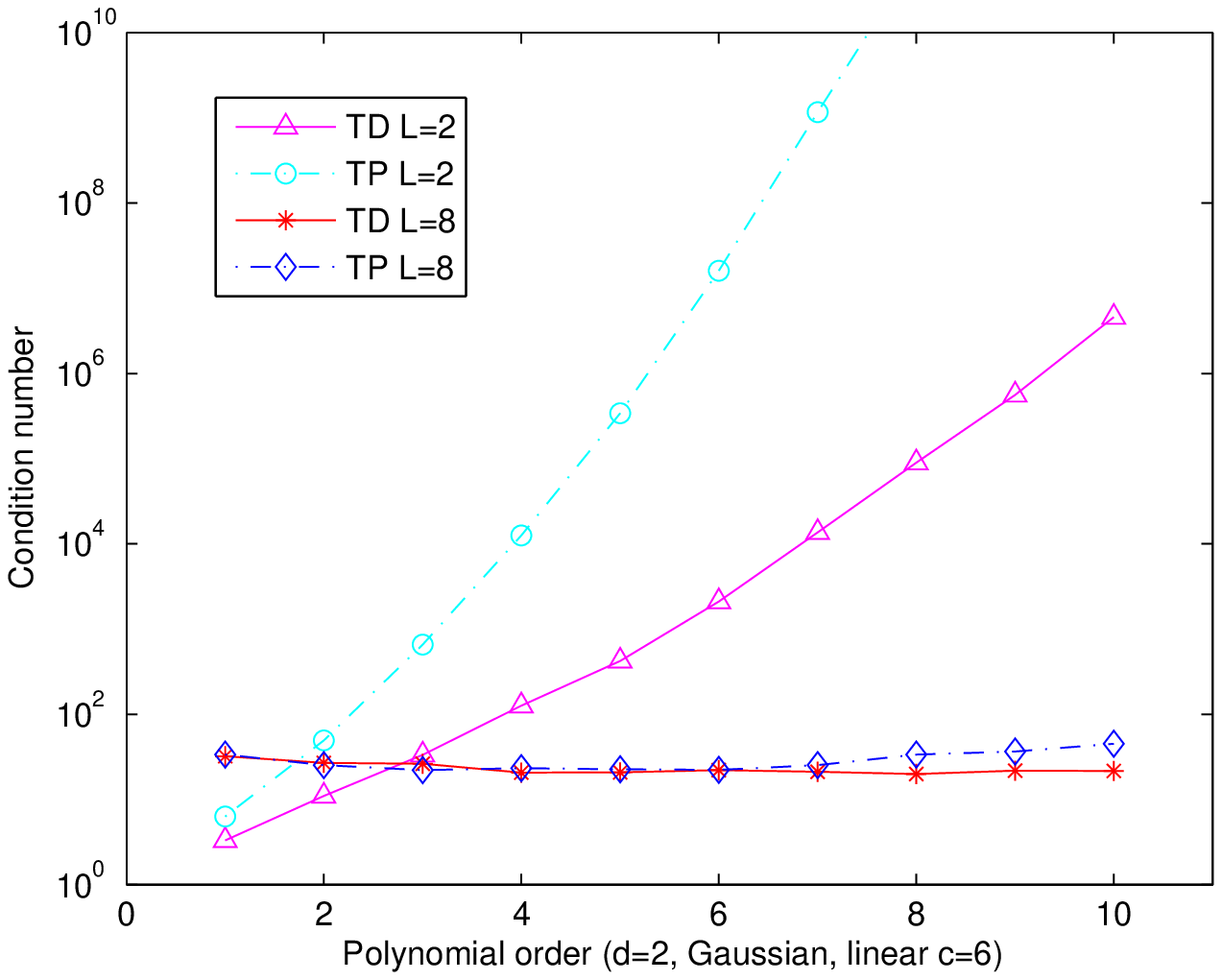}
\includegraphics[width=0.43\textwidth]{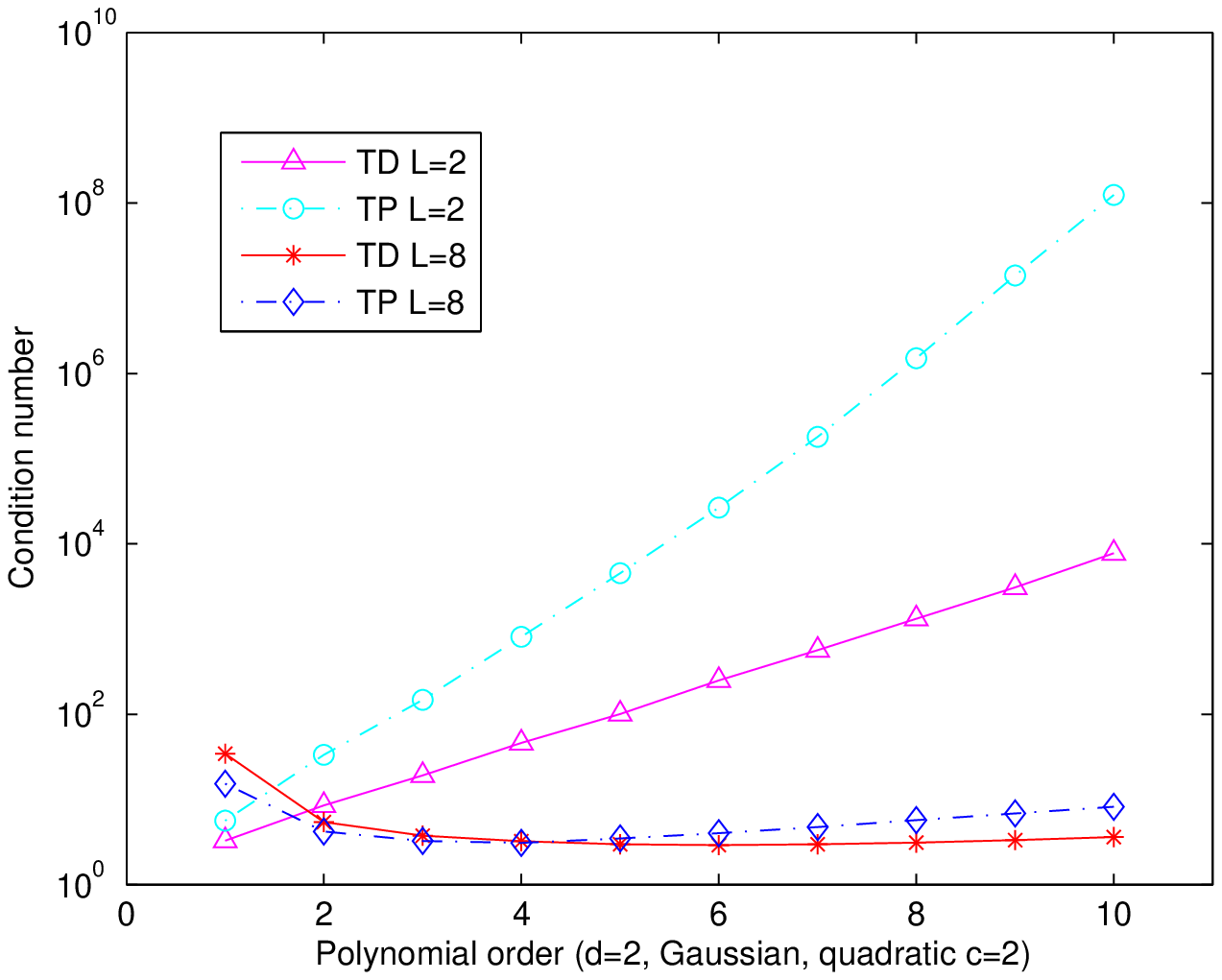}
\end{center}
\caption{\small
Condition numbers with respect to polynomial order.
Left is for $m=6* (\# \Lambda) $
and the right is for $m=4* (\# \Lambda)^2 $.
Top: 1D Gaussian, Bottom: 2D Gaussian.}
\label{fig:herfun_2d}
\end{figure}

\begin{figure}[h]
\begin{center}
\includegraphics[width=0.45\textwidth]{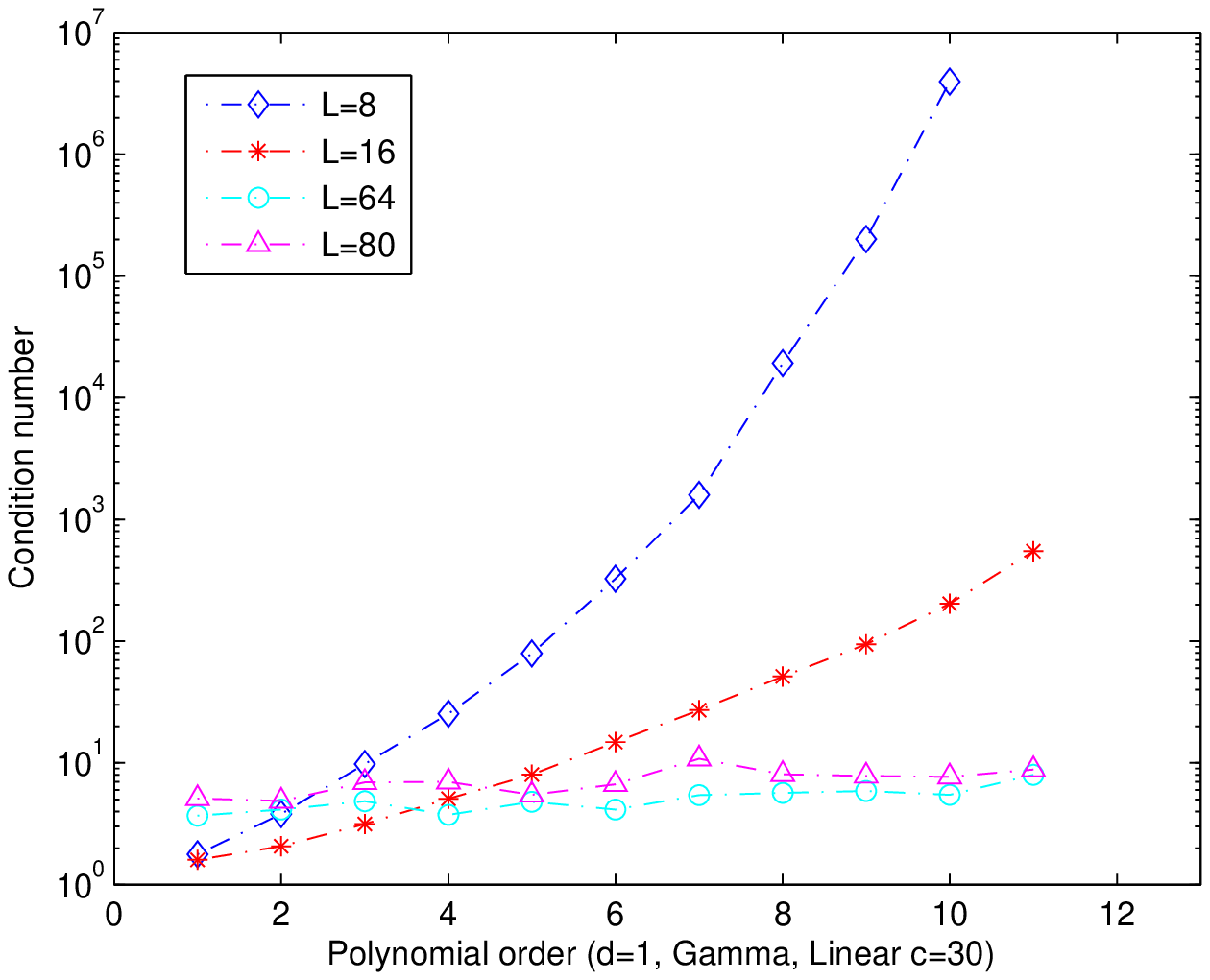}
\includegraphics[width=0.45\textwidth]{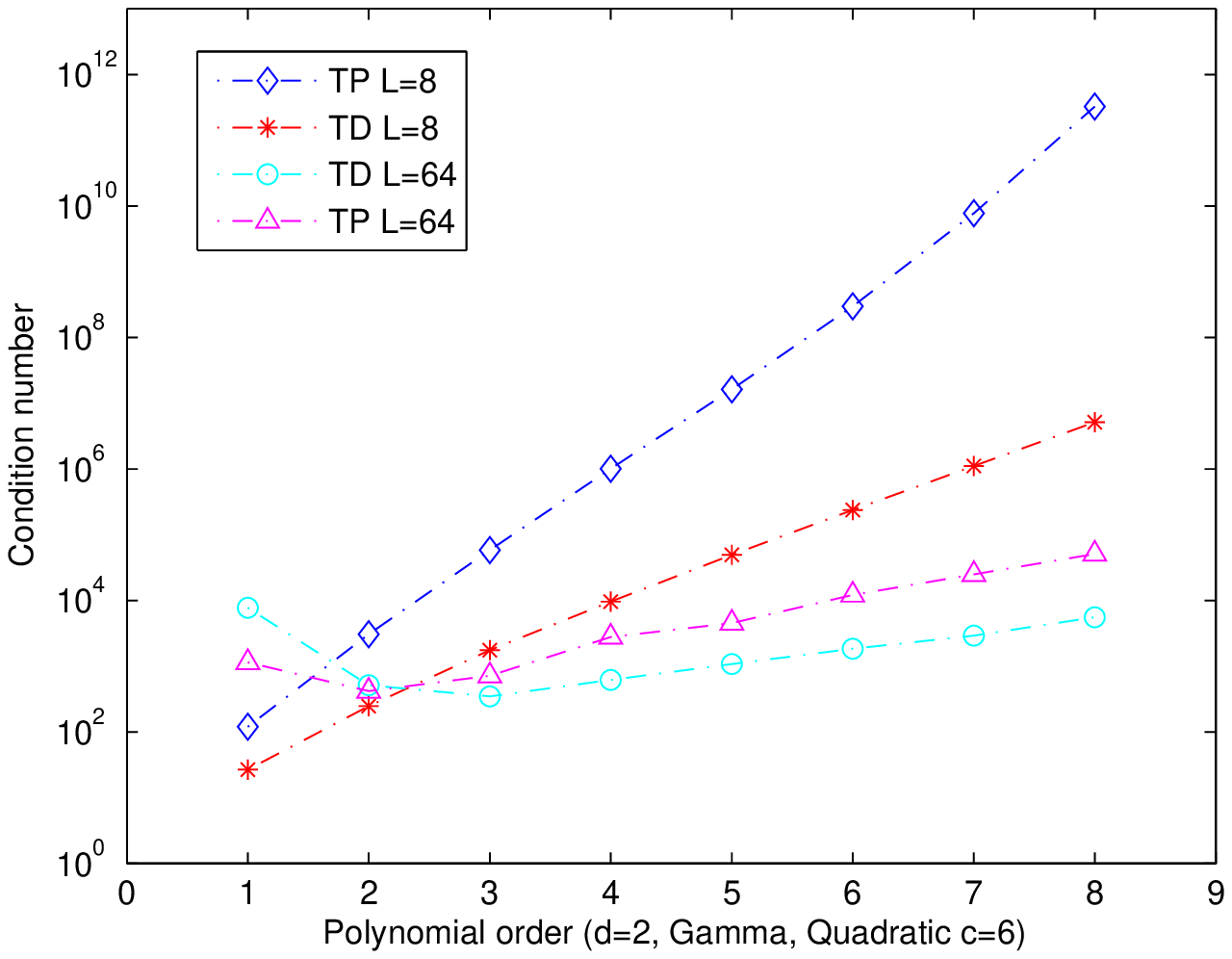}
\end{center}
\caption{\small
Condition numbers with respect to polynomial order.
Left: 1D Gamma, $m=30*(\# \Lambda) $.  Right:  2D Gamma,
$m=6*(\# \Lambda)^2 $.}
\label{fig:lagfun_condition}
\end{figure}

In Fig. \ref{fig:herfun_2d}, the condition numbers with respect to
the bases of order $q$ are given for one-dimensional Hermite function bases. The left plot is devoted to the linear rule with $m=6* (\# \Lambda) $, while the right plot is for the quadratic rule with $m=4*(\# \Lambda)^2$. In both cases, we can see that using a relatively large transform parameter $L,$ the random matrices $\mathbf{A}$ are well conditioned.  The two-dimensional cases are reported in the bottom of Fig.
\ref{fig:herfun_2d} for both the TP space and the TD space. Again, the parameter $L=8$ results in well conditioned design matrix, for both the TP and TD spaces. However, under the same parameter (say $L=2$), the design matrix of the TD spaces are much better conditioned than that for the TP spaces,
which is one of the reasons that
the TD space is  preferred for higher dimensional approximation.

Similar numerical tests are carried out for the Laguerre bases
and in this case the mapping (\ref{star1})  with $r=1$ is used.
The 1D result in the left of Fig. \ref{fig:lagfun_condition}
suggests that the parameter $L=8$ can no longer guarantee
the stability, while a larger parameter (say $L=64$)
will work. The two-dimensional
plot is  given in the right of the figure.
Again, the parameter $L=64$ results in a better condition number
for the design matrices.
We also note that more points and larger parameters $L$
are needed for higher dimensional cases.  Moreover,
the TD space ( $\circ$ and $\ast$ plots) provide
better stability than that of the TP
( $\triangleleft$ and $\diamond $ plots) space.

We conclude that the design matrix $\mathbf{A}$ can be
well-conditioned under a set of transformed random points
with some relatively large parameter $L$.
As the decay rate for Gaussian is faster than that
for Laguerre, the transformation parameter  $L$ for the Gaussian must be smaller than that for the Leguerre function

In the following, a rigorous analysis for the
stability will be provided.
We will only provide the proof for the one-dimensional Hermite
functions case; the proof can be extended
to the Laguerre case in a straightforward manner.

We first give a lemma concerning the decay properties of the  Hermite functions.

\begin{lemma}
\label{lemma31}
For any integer $K,$ we can find a constant $\tau>0$ such that
\begin{equation}
|\tilde{H}_k(y)| \leq |y|^{-\frac{3}{2}}, \quad \forall \, 0\leq k \leq K-1,
\end{equation}
provided that $|y|>\tau.$
\begin{proof}
Such a simple result is true because for any $t>0$ we have
\begin{equation}
|\tilde{H}_k(y)| \cdot |y|^t \rightarrow 0 \quad \textmd{when} \quad |y| \rightarrow \infty,
\end{equation}
due to the involvement of the
factor $e^{-\frac{y^2}{2}}$ in the Hermite functions.
\end{proof}
\end{lemma}
We are now ready to prove the stability. Such analysis requires an understanding of how the scaled random matrix $\mathbf{\hat{A}}=L\mathbf{A}$ deviates from its expectation $\mathbb{E}[\mathbf{\hat{A}}]$ in
probability $\mathbf{Pr}\{\cdot\}$. Note that the matrix $\mathbf{\hat{A}}$ can be written as
\begin{align*}
\mathbf{\hat{A}}= \mathbf{X}_1+ \mathbf{X}_2+ \cdot\cdot\cdot+ \mathbf{X}_m,
\end{align*}
where the $\mathbf{X}_i$ are i.i.d. copies of the random matrix
\begin{equation}
\label{satr4}
\mathbf{X}= \frac{L}{m} \left( \tilde{H}_i(y) \tilde{H}_j(y)\right)_{i,j=0,...,K-1},
\end{equation}
where $y$ is a transformed uniform random variable. We now state the stability result

\begin{theorem}
The least square approach using the Hermite functions (\ref{hf})
and the transformed uniform random points (\ref{star1})
is stable in the sense that the scaled design matrix satisfies
that $\forall\, r>0$
\begin{equation}\label{eq:stability}
\mathbf{Pr}\left\{||| \mathbf{\hat{A}}-\mathbf{I} |||\geq
\frac{5}{8} \right\} \leq 2m^{-r},
\end{equation}
provided that
\begin{equation}
\label{star3}
K \leq \kappa \frac{m}{\log m} \quad  \textmd{with} \quad
\kappa: = \frac{4 c_{1/2}}{3(1+r)},
\quad c_{\frac{1}{2}}=\frac{1}{2} +\frac{1}{2}\log\frac{1}{2} >0,
\end{equation}
and the mapping parameter $L$ in (\ref{star1})
satisfies
\begin{equation} \label{star2}
L > \max\{3\tau, 5\sqrt{K}\},
\end{equation}
where $m$ is the number of the random points, $K$ is the degree of the polynomial, and $M$ is the constant
given in Lemma \ref{lemma31}.
\end{theorem}

\begin{proof}
The analysis follows closely \cite{Cohen} and will use the following Chernoff bound \cite{bounds_I,bounds_II}.
Let  $\mathbf{X}_1,...,\mathbf{X}_n$ be independent $K\times K$ random self-adjoint and positive matrices satisfying
\begin{align*}
\lambda_{max}(\mathbf{X}_i)= |||\mathbf{X}_i||| \leq R
\end{align*}
almost surely, and let
\begin{align*}
\mu_{min}:= \lambda_{min}\left(\sum_{i=1}^m \mathbb{E}\left[\mathbf{X}_i\right]\right), \quad \mu_{max}:= \lambda_{max}\left(\sum_{i=1}^m \mathbb{E}\left[\mathbf{X}_i\right]\right).
\end{align*}
Then, one has for $0< \delta< 1$
\begin{align}\label{eq:min}
\textmd{Pr}\left\{\lambda_{min}\left(\sum_{i=1}^m \mathbf{X}_i\right) < (1-\delta)\mu_{min}\right\}\leq K \left(\frac{e^{-\delta}}{(1-\delta)^{1-\delta}}\right)^{\mu_{min}/R},
\end{align}
\begin{align}\label{eq:max}
\textmd{Pr}\left\{\lambda_{max}\left(\sum_{i=1}^m \mathbf{X}_i\right) > (1+\delta)\mu_{max}\right\}\leq K \left(\frac{e^{\delta}}{(1+\delta)^{1+\delta}}\right)^{\mu_{max}/R}.
\end{align}
Note that a rank 1 symmetric matrix $ab^T = (b_ja_k)_{j,k=1,...,m}$ has its spectral norm equal to the product of the Euclidean norms of the vectors $a$ and $b,$ and therefore we have
\begin{align*}
|||\mathbf{X}_i||| \leq \frac{1}{m}\sum_{i=0}^{K-1} \tilde{H}^2_i = \frac{ M(K)}{m}  :=  R \quad
\textmd{with} \quad
M(K) = \sup_{y\in \mathbb{R}} \sum_{i=0}^{K-1} \tilde{H}^2_i(y).
\end{align*}
We are now at the stage to find $\mu_{min}$ and $\mu_{max}.$ Let
\begin{align*}
\mathbf{\bar{A}}=\mathbb{E}[\mathbf{\hat{A}}]=\sum_{i=1}^m \mathbb{E}\left[\mathbf{X}_i\right].
\end{align*}
Using the definition of the expectation and Eq. (\ref{eq:mapping}), we know that the elements of $\mathbf{\bar{A}}$ satisfy
\begin{align*}
a_{i,j}= \int_{-1}^{1} L \tilde{H}_i\big(y(\xi)\big) \tilde{H}_j\big(y(\xi)\big) d\xi  =\int_{-\infty}^{+\infty} \left(1-\textmd{tanh}^2\left(\frac{y}{L}\right)\right) \tilde{H}_i( y) \tilde{H}_j( y) dy.
\end{align*}
Let $V:=\textmd{span}\big\{\tilde{H}_0,...,\tilde{H}_{K-1}\big\},$ and
\begin{align*}
a(u,v)=\int_{-\infty}^{+\infty} \left(1-\textmd{tanh}^2\left(\frac{y}{L}\right)\right) u v dy,  \quad b(u,v)=\int_{-\infty}^{+\infty} u v dy.
\end{align*}
By the Rayleigh quotient argument \cite{Rayleigh}, we have
\begin{equation}\label{eq:relaiy}
\mu_{min} = \min_{v\in V} \frac{a(v,v)}{b(v,v)}, \quad
\mu_{max} = \max_{v\in V} \frac{a(v,v)}{b(v,v)}.
\end{equation}
It is easy to verify that
\begin{align}
\mu_{max} = \max_{v\in V} \frac{a(v,v)}{b(v,v)} \leq 1.
\end{align}
We now estimate $\mu_{min}.$
Let $v=\sum_{k=0}^{K-1} c_k \tilde{H}_k$. We have
\begin{align}\label{eq:estimatesI}
a(v,v)\geq \left(1-\textmd{tanh}^2(\frac{1}{3})\right) \int_{-\frac{L}{3}}^{\frac{L}{3}}\!\! v^2 dy
=\left(1-\textmd{tanh}^2(\frac{1}{3})\right) \left( \int_{-\infty}^{\infty}\!\! v^2 dy-2\varepsilon\right),
\end{align}
where
\begin{eqnarray}\label{eq:estimatesII}
\varepsilon &=&\int_{\frac{L}{3}}^\infty\! v^2 dy  = \int_{\frac{L}{3}}^\infty\! \left(\sum_{k=0}^{K-1} c_k \tilde{H}_k\right)^2 dy
\nn \\
&\leq& K^2 \max_i\{c^2_i\} \int_{\frac{L}{3}}^\infty  y^{-3} dy
\leq \frac{3^4 K^2}{4 L^4} \left(\sum_{k=0}^{K-1} c^2_k\right),
\end{eqnarray}
where we have used Lemma \ref{lemma31} with $L\geq 3\tau.$
If $L > \max\{3\tau, 5\sqrt{K}\}$, then
using Eqs. (\ref{eq:relaiy}), (\ref{eq:estimatesI})
and (\ref{eq:estimatesII}) gives
\begin{equation}
\mu_{min} \geq \left(1-\textmd{tanh}^2(\frac{1}{3})\right)
\left(1- \frac{3^4 K^2}{2 L^4}\right) \geq \frac{3}{4}.
\end{equation}
Putting $\mu_{min}=\frac{3}{4}$ and $\mu_{max}=1$ into
Eqs. (\ref{eq:min}) and (\ref{eq:max}) respectively
and letting $\delta=\frac{1}{2}$ yield
\begin{align}
 \textmd{Pr}\left\{ ||| \mathbf{\hat{A}}-\mathbf{I}||| > \frac{5}{8} \right\} \leq 2K\left(\frac{e^{-{\frac{1}{2}}}}{(\frac{1}{2})^{1/2}}\right)^{3/4R} = 2K\exp \left( -\frac{4 c_{1/2} m}{ 3 M(K)}\right),
\end{align}
where $c_{\frac{1}{2}}=\frac{1}{2} +\frac{1}{2}\log\frac{1}{2} >0.$
Finally letting
\begin{equation}
\label{3.1}
M(K) \leq \kappa \frac{m}{\log m} \quad  \textmd{with} \quad \kappa: = \frac{4 c_{1/2}}{3(1+r)}
\end{equation}
yields the desired result (\ref{eq:stability}).
Note that we have $\tilde{H}^2_i(y) < 1, 0\le i\le K-1.$
Consequently, we can choose $M(K)=K$ in (\ref{3.1}).
The proof is complete.
\end{proof}

Note that the requirement (\ref{star2}) for $L$ may
not be optimal, as the numerical tests  in Fig.
\ref{fig:herfun_2d} suggest
 that the mapping with parameter $L \geq 8$ results in
very stable approach (up to polynomial order of 25).
In fact, inspired by the above proof, we need to choose
a large parameter $L$ so that the integral
(\ref{eq:estimatesII}) is sufficiently small. On the other hand,
it is known that the largest root of $\tilde{H}_K$
behaves like $\sqrt{2K},$  so the requirement $L>\sqrt{5K}$
asymptotically coincides with that $L$ should be bigger than the largest
root of $\tilde{H}_K$.

We also point out that the proof above can be extended to the Laguerre
case. However, as the Laguerre functions decay much slower
than the Hermite functions, a larger parameter $L$
(approximately the square of the Hermite case) should be used.
This can also be estimated by noting that the largest root
of $\tilde{L}_K$ behaves like $cK.$  Again, these theoretical
results are in good agreement  with our numerical tests in Fig.
\ref{fig:lagfun_condition}, where $L=8^2=64$ leads to
very stable approach under the linear rule.

\subsection{Convergence and the scaling factor: motivation}

\begin{figure}[h]
\begin{center}
\includegraphics[width=0.45\textwidth]{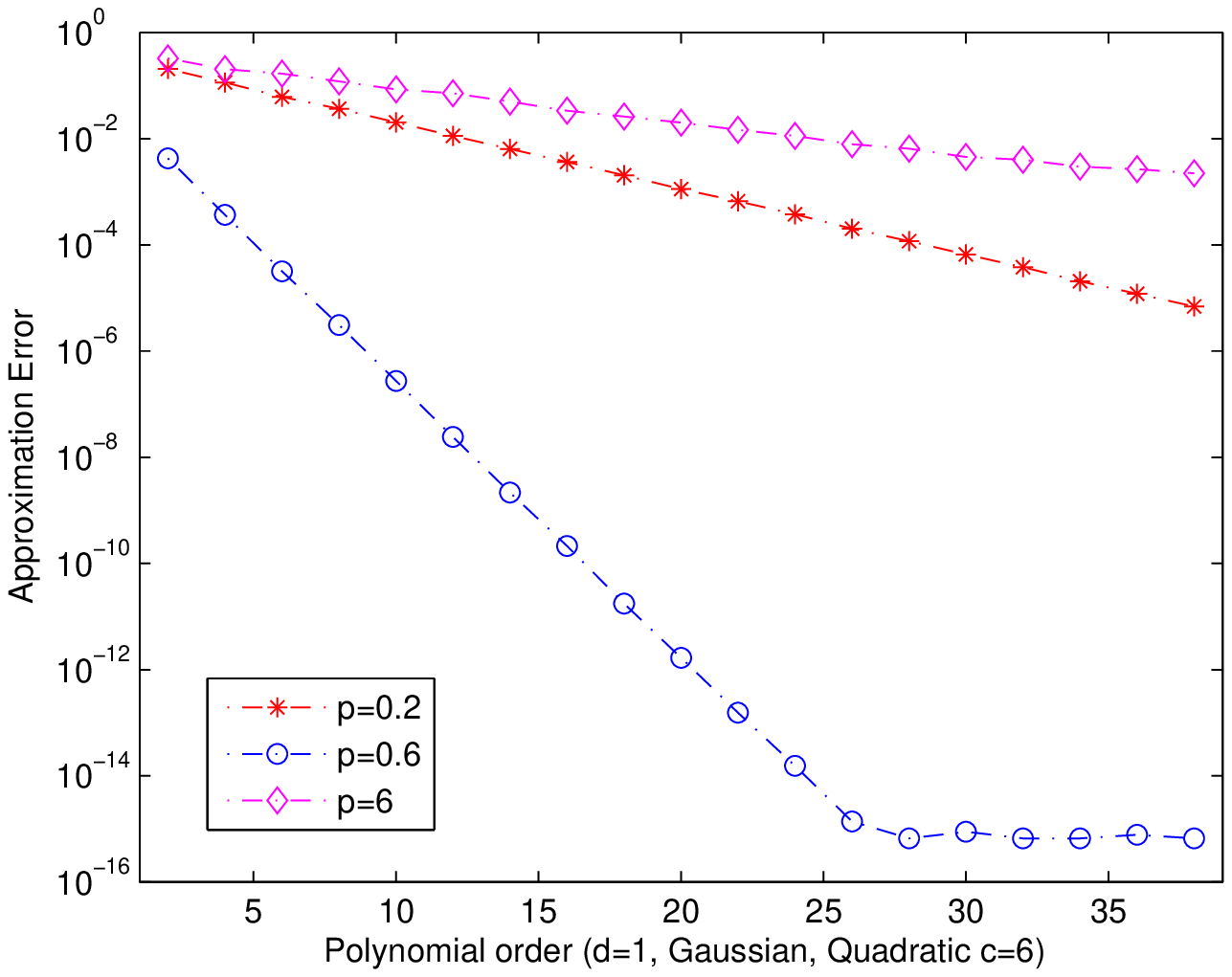}
\includegraphics[width=0.45\textwidth]{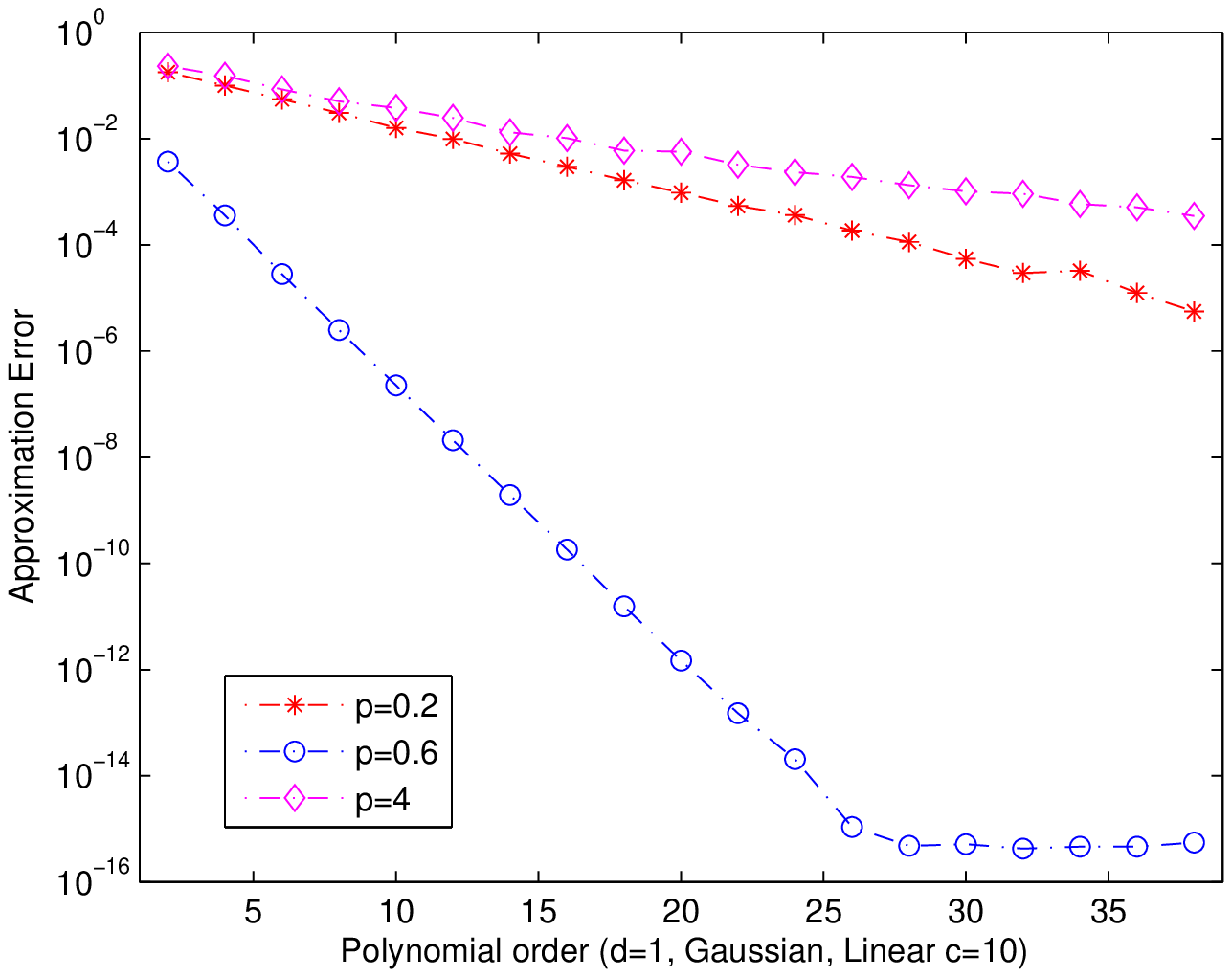}
\end{center}
\caption{\small
Approximation error for $f(y)= 2^{-p y^2}$ against the polynomial order with different parameter $p$. Left: quadratic rule
$m=c*(\# \Lambda)^2 $ with $c=6$. Right: linear rule with $c=10$. }\label{fig:AP_test}
\end{figure}

In this subsection, we shall investigate the convergence issue.
By the discussions in the last section, we know that we can use
the transformation parameter $L$ to obtain a
stable approach. Furthermore, inspired by the proof in
\cite{Cohen,Cohen2014}, one can expect the following convergence property of the least square approach
\begin{equation}
\mathbf{Pr}\left\{ ||f-f_m||_\rho \geq C \min_{v\in V} ||f-v||_{L^\infty(\mathbb{R})}  \right\} \leq 2m^{-r},
\end{equation}
with suitable norm $||\cdot||_\rho$ associate with the transformation $\rho(y)=1-\tanh^2(\frac{y}{L}),$ where $f_m$ is the least square solution.  As the proof follows directly the framework of \cite{Cohen2014}, and thus is omitted here. Although the above results implies the error estimate in the finite space $V,$ from the convergence point of view the rate of convergence ($\min_{v\in V}||f-v||$), may depend strongly on properties of the underlying function, such as the regularity and the decay rate.

To this end, we first demonstrate some numerical results for approximating
the function $f(y)=2^{-py^2}$ with a Gaussian parameter $y$ and a
constant $p$. In the following experiments, we will report
the error in the $L^\infty$ norm. More precisely, we compute the
maximum error on 4000 random grids in $\mathbb{R}.$
The approximation error using the Hermite functions against the polynomial order is given in Fig. \ref{fig:AP_test}.
In the computations,
the parameter $L=8$ is used to guarantee the stability.
It can be seen from Fig. \ref{fig:AP_test} that both the
linear rule $m=c(\# \Lambda)$ (Right) and the quadratic
rule $m=c(\# \Lambda)^2$ (Left) produce
very stable approach up to degree $q=38$.

Another simple observation is that although the function $f$ is sufficiently smooth for any values of $p,$  the convergence rate differs dramatically for $p.$ For $p=0.6$ ($\circ$ plot), the convergence is very fast, while for $p=0.2$ or $p=4$, the convergence is very slow (yet, still stable). This is due to the use of the Hermite functions which behave approximately
like $\textmd{e}^{-y^2/2}$ at infinity.  It is noted that when the approximated function $f(y)$ matches such a decay property (e.g. $p=0.6$ which is close to 0.5), the convergence is fast, while the convergence is very slow when the approximated function decays much faster or much slower than the Gaussian function (e.g.,
$p=0.2$ or $4$).

A remedy to fix the above problem is the use of the so-called
scaling factor  \cite{SWT,Scaling}.
In spectral methods, the scaling factor is often used to
speed up the convergence for approximating functions that decay
fast in infinity. Such an idea was successfully applied to
the studies of different problems \cite{Sun,Yau,CPC13}.

We now introduce the basic idea of the scaling factor.
To this end, let
$f(y)$ be a function that decay exponentially, namely,
\begin{equation}\label{eq:decay}
|f(y)|<\epsilon,  \quad \forall \;  |y|>M,
\end{equation}
where $0< \epsilon\ll 1$ and $M>0$ are some constants.
The idea of using the scaling factor is to expand $f$ as
\begin{equation}\label{eq:herfun_p=6}
f(y)=\sum_{n=0}^{K-1} c_n \tilde{H}_n(\alpha y) \,\,\Leftrightarrow\,\,
f\left(\frac{y}{\alpha}\right)=\sum_{n=0}^{K-1} c_n \tilde{H}_n(y),
\end{equation}
where $\alpha >0$ is a scaling factor. The key issue of using
$\alpha$ is to scale the points $\{y_i\}$ so that ${y_i}/{\alpha}$
are well within the effective support of $f.$

To see the effect of the scaling, let us carry out some numerical tests. We first consider a fast
decay function $f(y)=2^{-6y^2}$.
In the top of
Fig. \ref{fig:herfun_p=0.5}, the maximum approximation error
with respect to polynomial order is shown for
one-dimensional case. In the left of the figure, we fix
the parameter $L=8$ to ensure stability.
It is noticed that the convergence for the
original Hermite function approach $(\alpha=1,$ $\circ$ plot)
is very slow (although stable), while the use of a scaling factor $\alpha$ indeed can significantly improve the convergence rate.
In this example, the optimal scaling factor seems to be
around $\alpha=2.8$ ($\ast$ plot). The right of the figure
presents the convergence properties using
the scaling $\alpha=2.8$ but with variate parameters $L.$
It is noticed that, under small parameters ($L=0.5$ or 1),
the convergence rate deteriorate when large polynomial
order is used. This  is due to the instability when small
parameters $L$ are used. In contrast,
the parameter $L=8$ ($\diamond$ plot) results in very stable approach.

Let us now consider a slowly decaying function
$\tilde{f}(y)=2^{-0.2y^2}$.
The corresponding convergence results are shown in the bottom of Fig. \ref{fig:herfun_p=0.5}. The bottom left uses the
fixed parameter $L=8$ and several values of $\alpha$.
It is noticed that the optimal scaling factor in this case is about $\alpha=0.82$ ($\ast$ plot) in terms of rate of convergence, although the results for all $\alpha$ are stable.
The bottom right shows  the error curves using the optimal scaling $\alpha=0.82$ but with various
parameters $L.$ It is noticed that with small parameters ($L=1$ or 2, $\ast$ and $\circ$ plots) the convergence rate deteriorates when large polynomial order is used. In contrast, the parameter $L=8$ ($\diamond$ plot) results in very stable approach.

\begin{figure}[t]
\begin{center}
\includegraphics[width=0.45\textwidth]{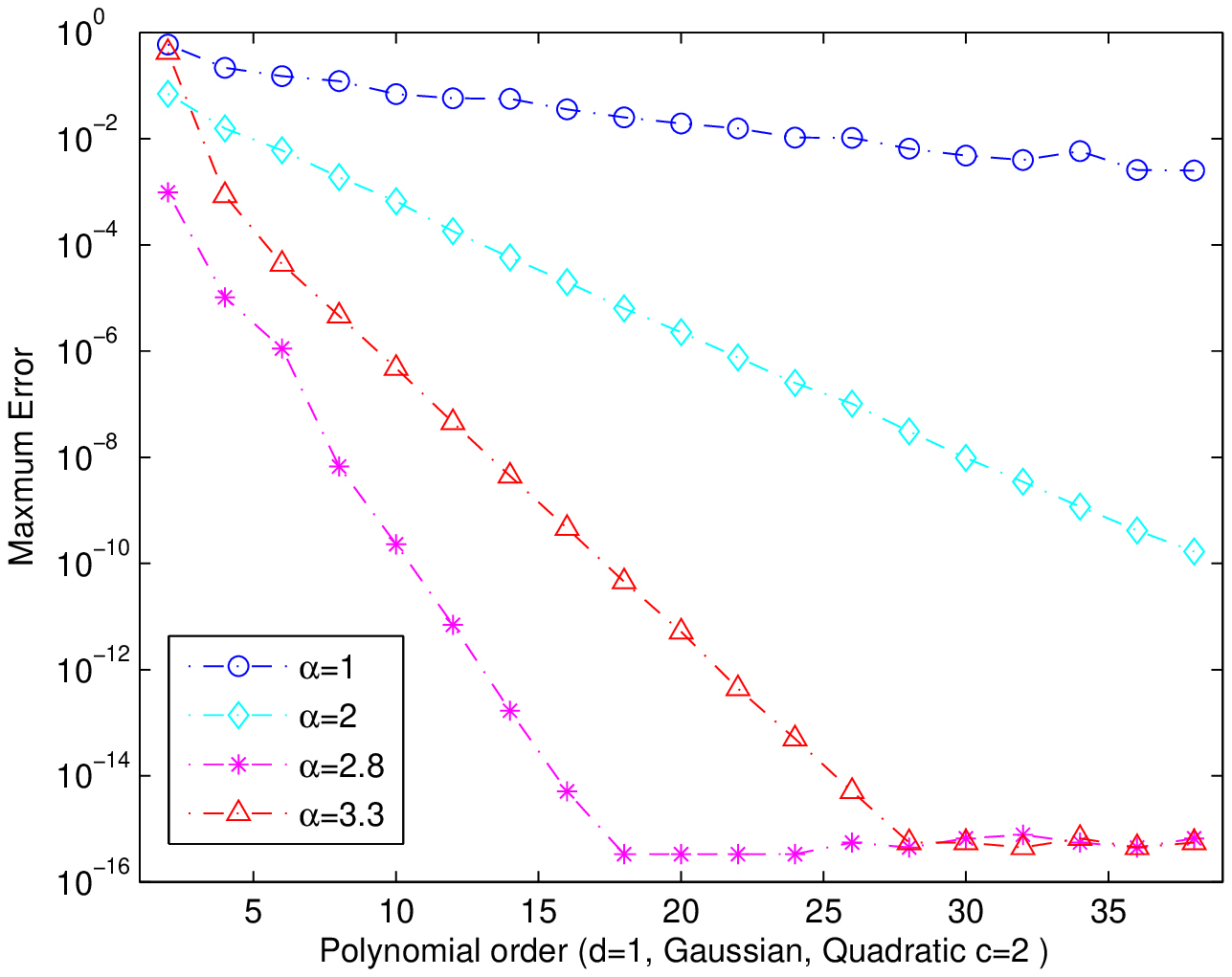}
\includegraphics[width=0.45\textwidth]{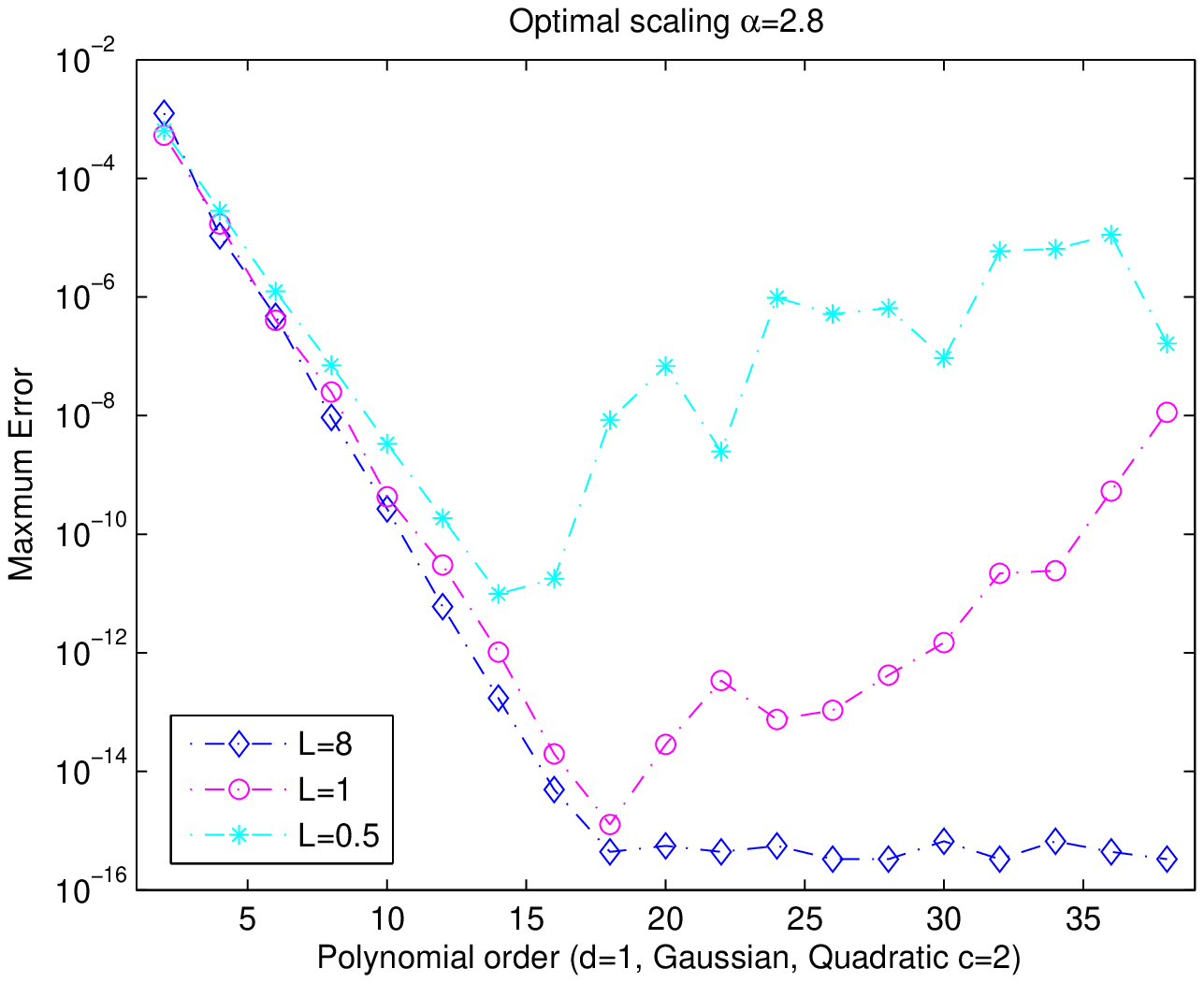}
\end{center}
\begin{center}
\includegraphics[width=0.45\textwidth]{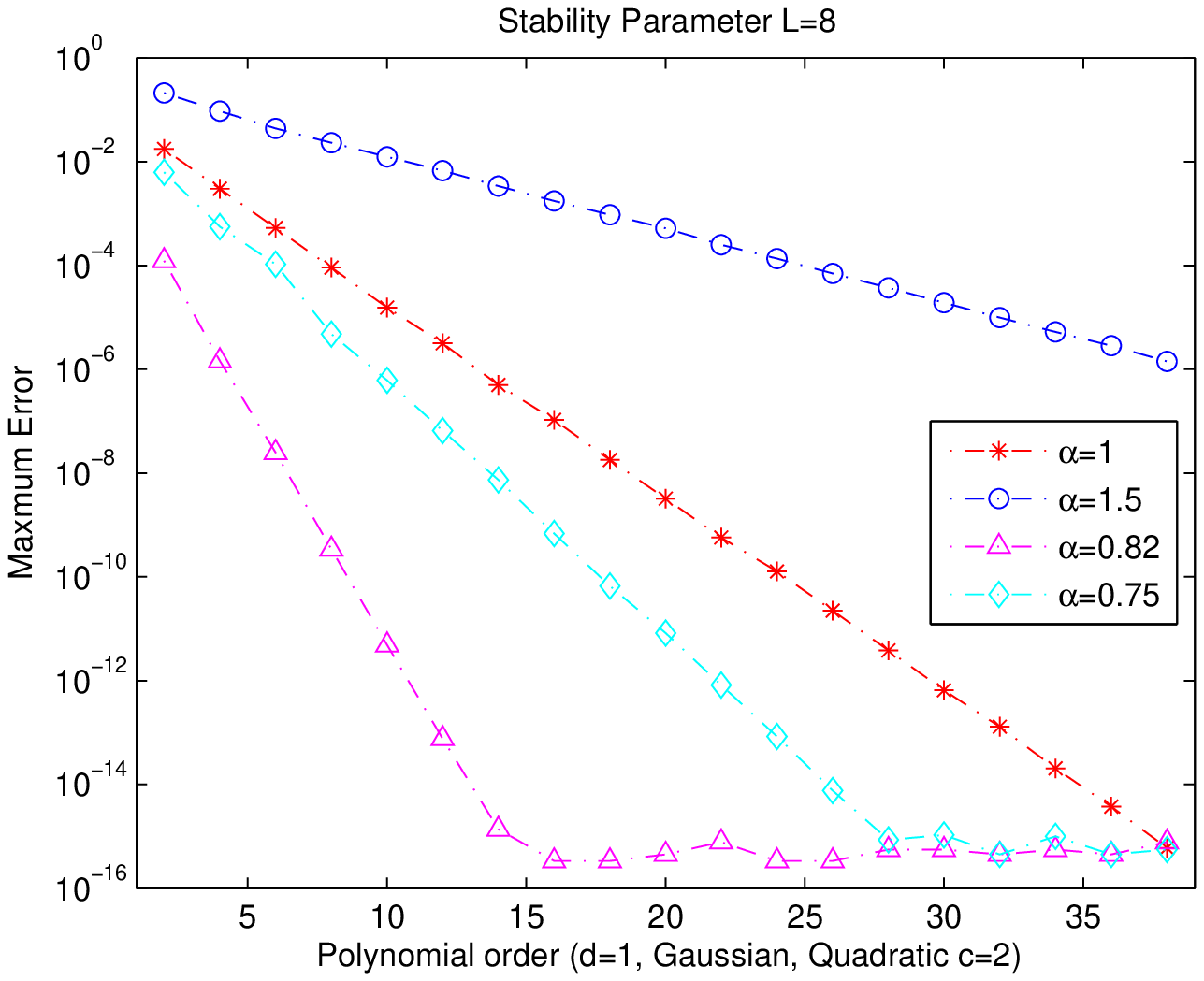}
\includegraphics[width=0.45\textwidth]{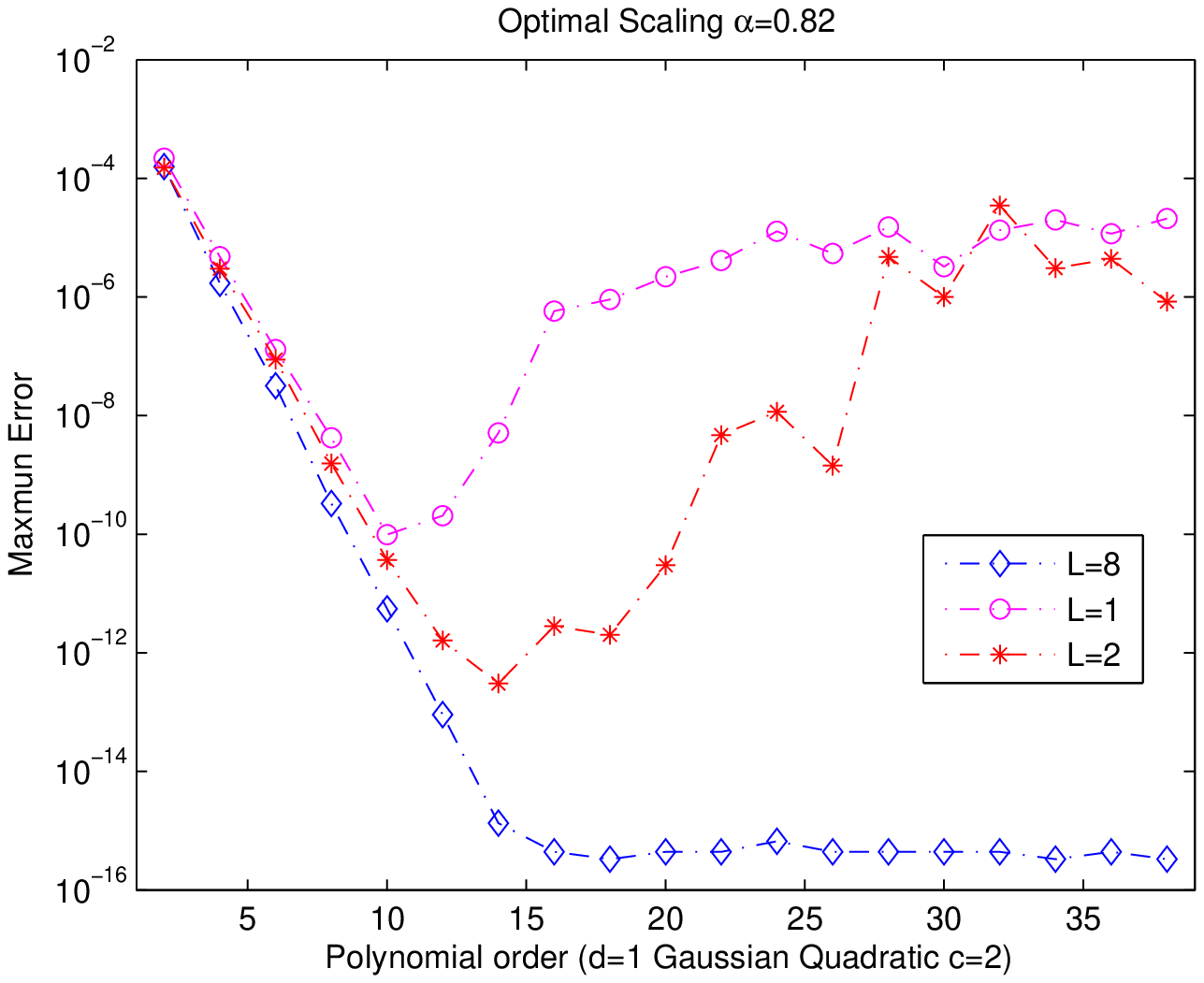}
\end{center}
\caption{\small
Convergence with respect to polynomial order ( 1D Gaussian
with $m=2*(\# \Lambda)^2$). Top: $f(y)=2^{-6 y^2}$. Left uses different scaling $\alpha$ with transformation parameter $L=8$, and right uses the  optimal scaling $\alpha=2.8$ with different parameter $L.$
Bottom: $f(y)=2^{-0.2 y^2}$. Left uses different scaling $\alpha$ with transformation parameter $L=8$, and right uses the  optimal scaling $\alpha=0.82$ with different parameter $L.$
}\label{fig:herfun_p=0.5}
\end{figure}

\subsection{Scaling factor: application to least square approach}

The above tests suggest that proper scaling factors should be employed to speed up the rate of convergence.  We now discuss how to find a feasible scaling in our least square approach. Note that the numerical solution (the expansion coefficients $\mathbf{c}$) satisfies
\begin{equation}
\mathbf{A} \mathbf{c} = \mathbf{f}
\end{equation}
with $\mathbf{A}$ being the design matrix, where
\begin{equation}
\mathbf{A}=\Big(\innerp{\tilde{H}_i, \tilde{H}_j}_m \Big)_{i,j=1}^N, \quad \mathbf{f}=\Big(\innerp{  f, \tilde{H}_j}_m \Big)_{j=1}^N.
\end{equation}
For ease of discussion, we assume that the points $\{y_i\}_{i=1}^m$
are in a absolute increase order, i.e.,
\begin{align*}
|y_1| \leq |y_2| \leq \cdot\cdot\cdot \leq |y_m|.
\end{align*}
Note that
 \begin{align}
\mathbf{f}_k= \left<f, H_k\right>_m = \sum_{i=1}^m f\!\left(\frac{y_i}{\alpha}\right) H_k(y_i), \quad k=1,...,N.
\end{align}

\begin{figure}[t]
\begin{center}
\includegraphics[width=0.45\textwidth]{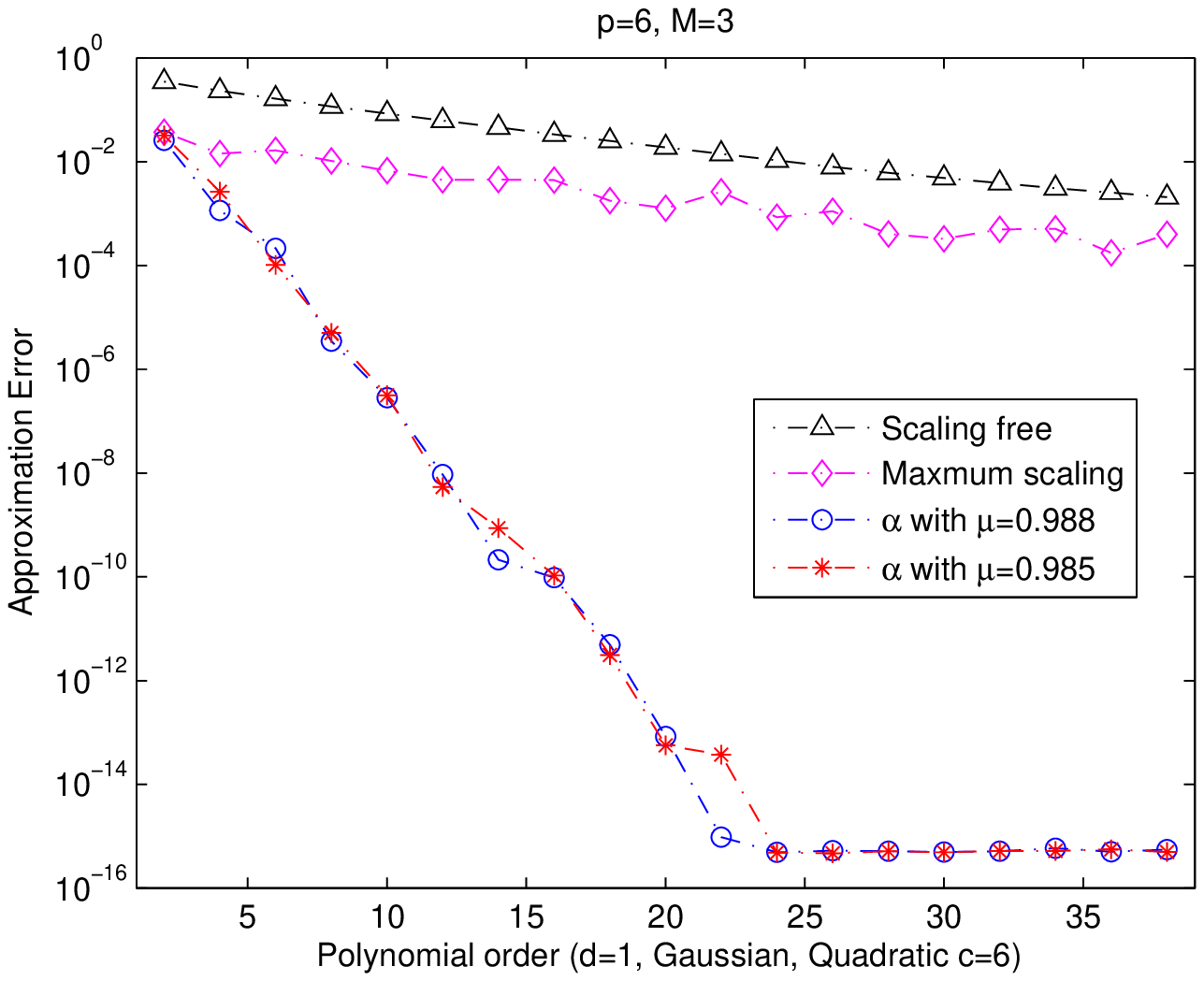}
\includegraphics[width=0.45\textwidth]{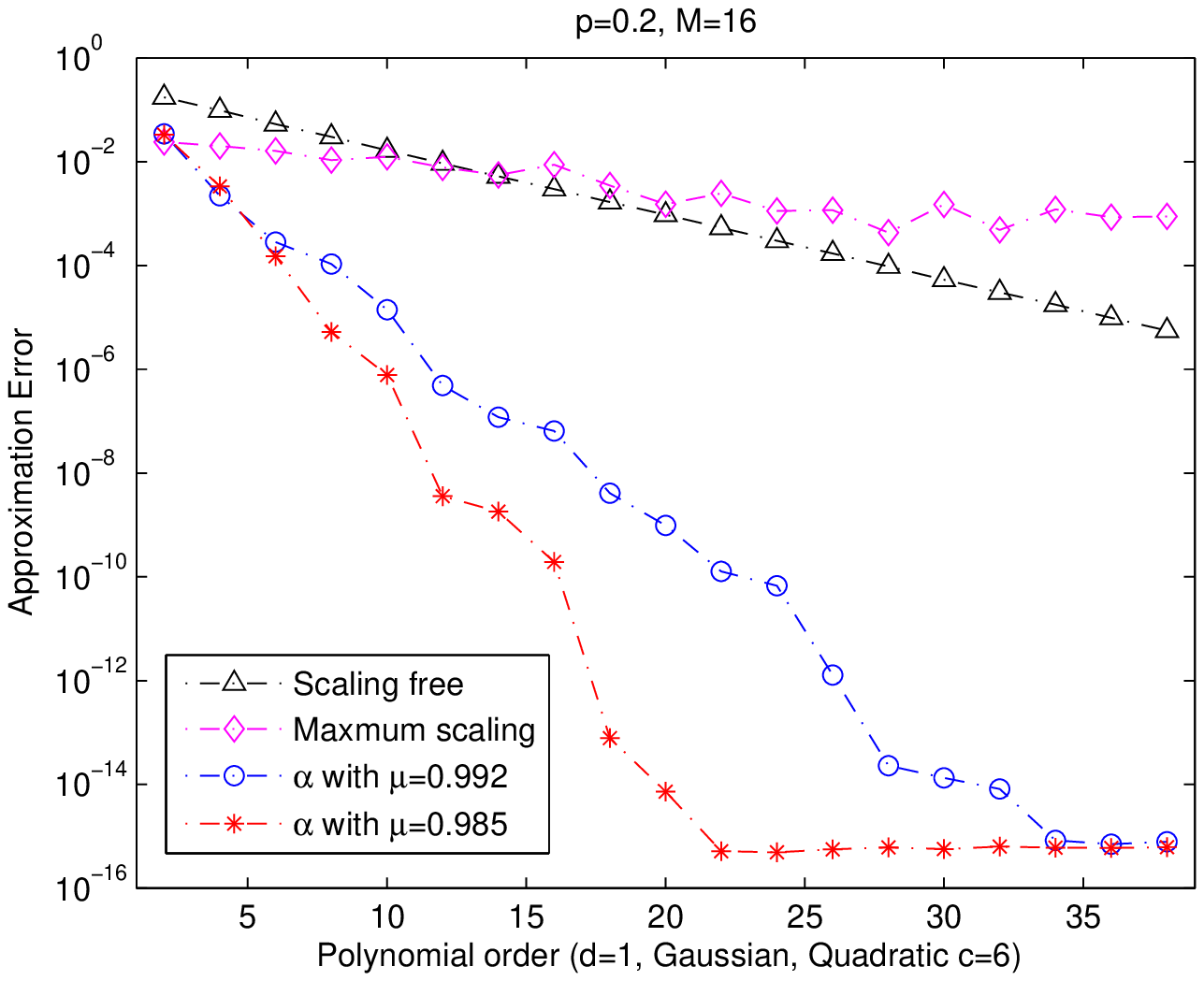}
\end{center}
\caption{\small
Error against polynomial order ($f(y)=2^{-py^2},$ 1D Gaussian, $m=6*(\# \Lambda)^2$). Left: $p=6$ with different scaling $\alpha.$ Right: $p=0.2$ with different scaling $\alpha.$}
\label{fig:herfun_scaling}
\end{figure}

\begin{figure}[h]
\begin{center}
\includegraphics[width=0.45\textwidth]{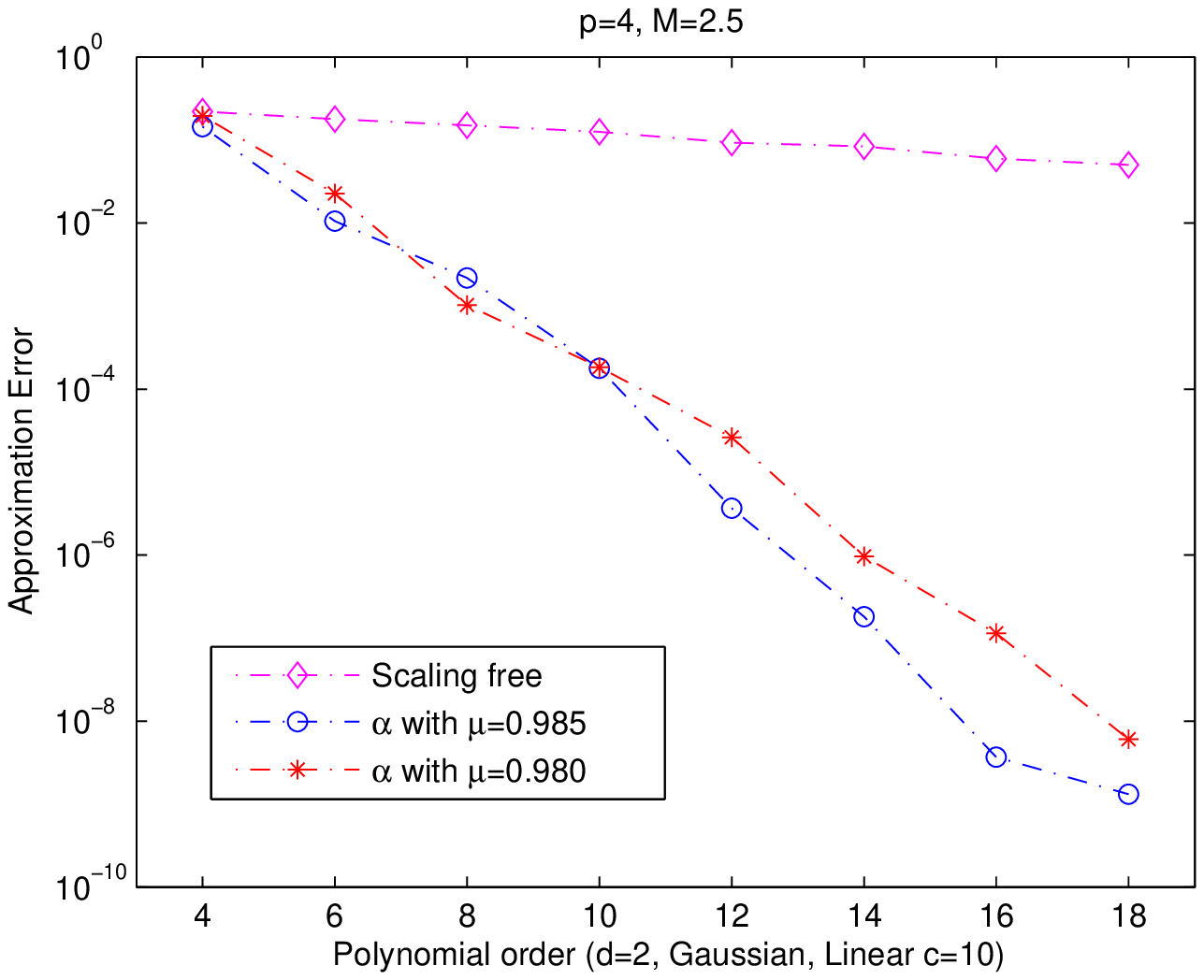}
\includegraphics[width=0.45\textwidth]{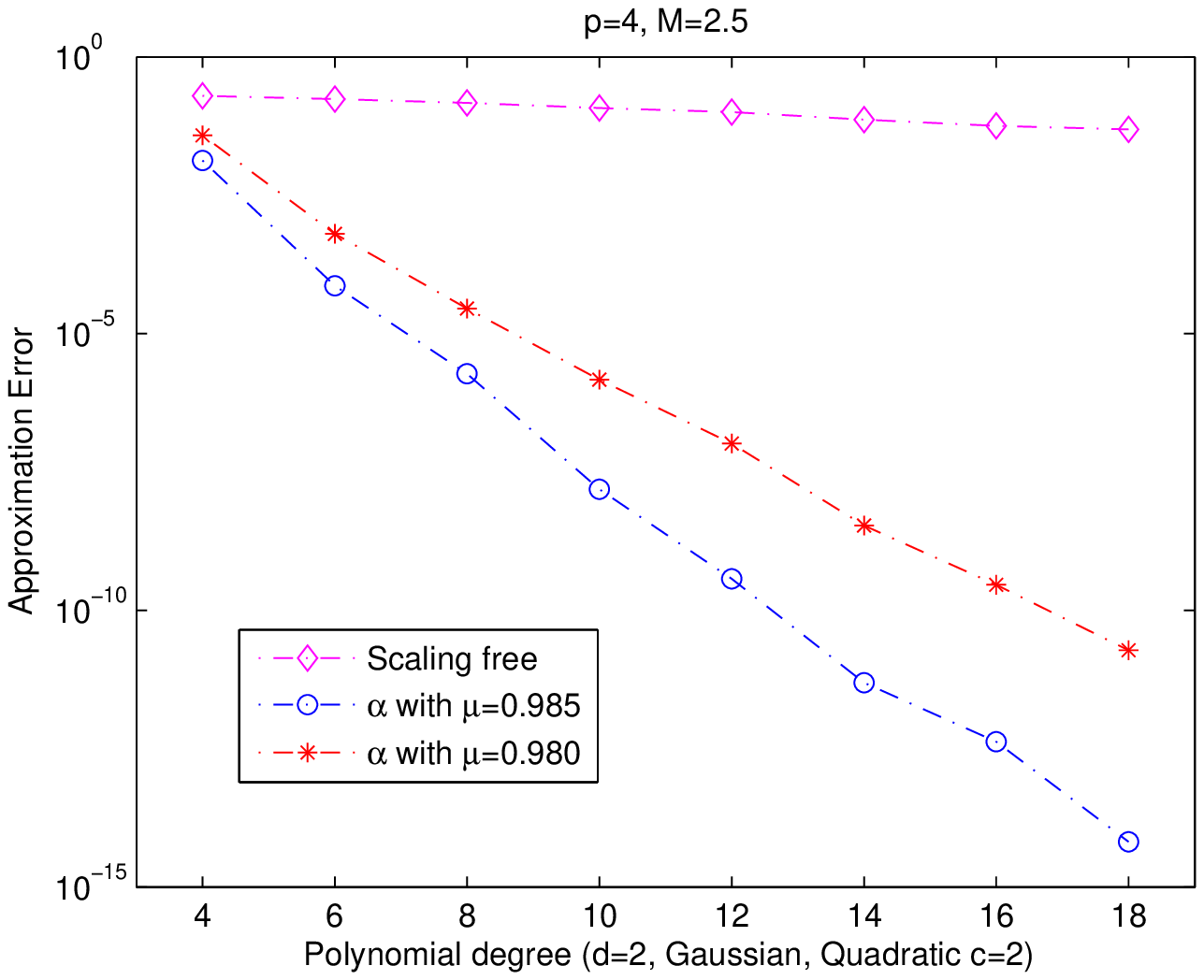}
\end{center}
\caption{\small
Numerical error against polynomial order ($\tilde{f}(y)=\textmd{e}^{-4(y_1^2+y_2^2)} \textmd{sin}(y_1+y_2),$ M=2.5, TD space.) Left: the linear rule with $m=10*(\# \Lambda)$ Right: the quadratic rule with $m=2*(\# \Lambda)^2.$}
\label{fig:herfun_2d_scaling}
\end{figure}

Clearly, in order to compute $\{c_k\}_{k=1}^N,$
we need to use information of $f$ from the interval $[-M,M]$ out of
which the contribution of $f$ is 0 in the sense of the
floating number. This observation suggests that
\begin{equation}\label{eq:guess}
\max_{1\leq j \leq m}\{|y_j|\}/\alpha \leq M \quad  \Rightarrow \quad \alpha = \max_{1\leq j \leq m}\{|y_j|\}/M.
\end{equation}
This idea is similar to the proposal given in \cite{Scaling}
in the context of pseudospectral methods. However, in our least
square approach the points $\{y_i\}_{i=1}^m$ are generated randomly.
The scaling $\alpha$ in (\ref{eq:guess}) may not be efficient from the
probability point of view: there is possibility that only few
points (may be only 2 or 3) are extremely large (we refer such points as \textit{bad points}), which means that the scaling
(\ref{eq:guess}) may \textit{over scale} the points.
This motivates us to {\em drop} such {\em bad points}.
More precisely, we choose
\begin{equation}\label{eq:quasi-optimal}
\tilde{\alpha} = \max_{1\leq j \leq \widetilde{m}}\{|y_j|\}/M,
\quad \tilde{m}= \lfloor \mu m \rfloor,
\end{equation}
where $\mu$ is a parameter close to 1. That is, we drop $m-\lceil \mu
m \rceil$ possible \textit{bad points}, and require
$\lfloor \mu m \rfloor$ points to contribute to  the computation of
$\{c_k\}_{k=1}^N.$
In practice, it is found that we can just set $\mu\sim 98\%,$,
 meaning that the probability of generating bad points is $2\%.$

We now repeat the numerical test in Fig.
\ref{fig:herfun_p=0.5} (the left ones), with particular attention
to the use of the scalings (\ref{eq:guess}) and (\ref{eq:quasi-optimal}).
The numerical results are given in Fig. \ref{fig:herfun_scaling},
where \textit{scaling free} stands for the results
without using a scaling, \textit{maxmum scaling} denotes the
scaling computed by (\ref{eq:guess}),
while \textit{scaling with $\mu=s\%$} means that the scaling
is computed by (\ref{eq:quasi-optimal}).
The left of Fig. \ref{fig:herfun_scaling} shows the convergence for approximating $\tilde{f}(y)=2^{-6y^2}$. In this case,
we simply set $M=3$, i.e., the effective support of $f(y)$ is chosen as $[-3,3]$. It can be seen that the numerical error with
scaling factor (\ref{eq:quasi-optimal}) decays  very fast
($\ast$ and $\circ$ plots) as compare to the \textit{scaling free}
case ($\triangleright$ plot),  while the results with
\textit{maxmum scaling} ($\diamond$ plot) behaves almost the same
as the \textit{scaling free} case .
The right plot is for $\tilde{f}(y)=2^{-0.2 y^2},$ and we
set $M=16$ for this test. A similar phenomenon is observed.

For high dimensional cases, a reasonable scaling should be chosen in
each direction. A two-dimensional test is provided in Fig.
\ref{fig:herfun_2d_scaling}. The function to be approximated is $\tilde{f}(y)=\textmd{e}^{-4(y_1^2+y_2^2)} \textmd{sin}(y_1+y_2),$ and the approximating space is the TD space. In the left plot, we have used the linear rule  $m=10*(\# \Lambda),$ while the quadratic rule with $m=2*(\# \Lambda)^2$ is used in the right plot. The scaling factors are computed by (\ref{eq:quasi-optimal}) with $M=2.5.$ It is shown that the convergence is stable, and the scaling works very well. Furthermore, it is noticed that the convergence rate of the quadratic rule (right) is better than the linear rule. This might be due to that the linear rule uses less points than the quadratic rule.

\begin{remark}
{\em
In practice, finding the optimal parameter $M$ is
not straightforward due to the limit information of the function $f$. Nevertheless, we can always find a reasonable $M$ if the information for $f$ is reasonably sufficient.
It remains a research issue on how to find acceptable $M$
if only a few evaluations of $f$ are available;
we will leave this problem for future studies.
}
\end{remark}

\section{Parametric UQ: illustrate examples}
\setcounter{equation}{0}

In this section, we discuss the application of the least square Hermite (Laguerre) approximations to parametric uncertainty analysis, precisely, we shall use the least square approach based on the Hermite (Laguerre) functions to compute the QoI of UQ problems.

\subsection{A simple random ODE}
We first consider a simple random ODE problems with Gamma random input:
\begin{equation} \label{ode}
\frac{d f}{d t} = - k(y) f,\quad
f(0)=1,
\end{equation}
where $k(y)$ is a function with respect to a random Gamma parameter $y.$ Note that for such problems with Gamma random input, the Laguerre functions will be used as the bases. To illustrate the idea, we set $k(y)=\beta y.$ We are interested in the second moment of the solution, i.e.,
$$
\textmd{QoI}=\int_{{\mathbb{R}}_+} \! \textmd{e}^{-y} f^2(t,y) dy.
$$
Note that in the least square approach, for each random point $y_i,$ one has to solve the ODE to get the information $f(t, y_i).$ The random points that located in $(0,\infty)$ used here is the transformed uniform random points. We will use the mapping (\ref{eq:mapping}) with parameters $r=1$ and $L=64$ to guarantee the stability. The numerical convergence results are shown ($t=1$) in Fig. \ref{fig:lagfun_scaling}. The left plots are for $\beta=1.5.$ Note that we are in fact approximating the function $\tilde{f}=\textmd{e}^{-y} f^2(t,y)=\textmd{e}^{-4yt}.$ It is noticed from Fig. \ref{fig:lagfun_scaling} that the convergence is very slow without using a scaling, and this is again due to the fast decay of $f$ compared to the Gamma measure. In this test, both the \textit{maximum scaling} and the \textit{scaling with} $\mu=s\%$ work well, which is different with the observations for the Gaussian measure. It is likely due the slow decay of the Gamma measure, which results in a very big effective support
(outside of the effective support $\tilde{f}$ is 0 with the machine accuracy) , and thus, the probability of \textit{over scale} is not so large as in the Gaussian case. The right plot is for $\beta =-0.65.$ Again, all scaling values work well,
although the scaling computed by (\ref{eq:quasi-optimal}) behave more stable.

\begin{figure}[h]
\begin{center}
\includegraphics[width=0.45\textwidth]{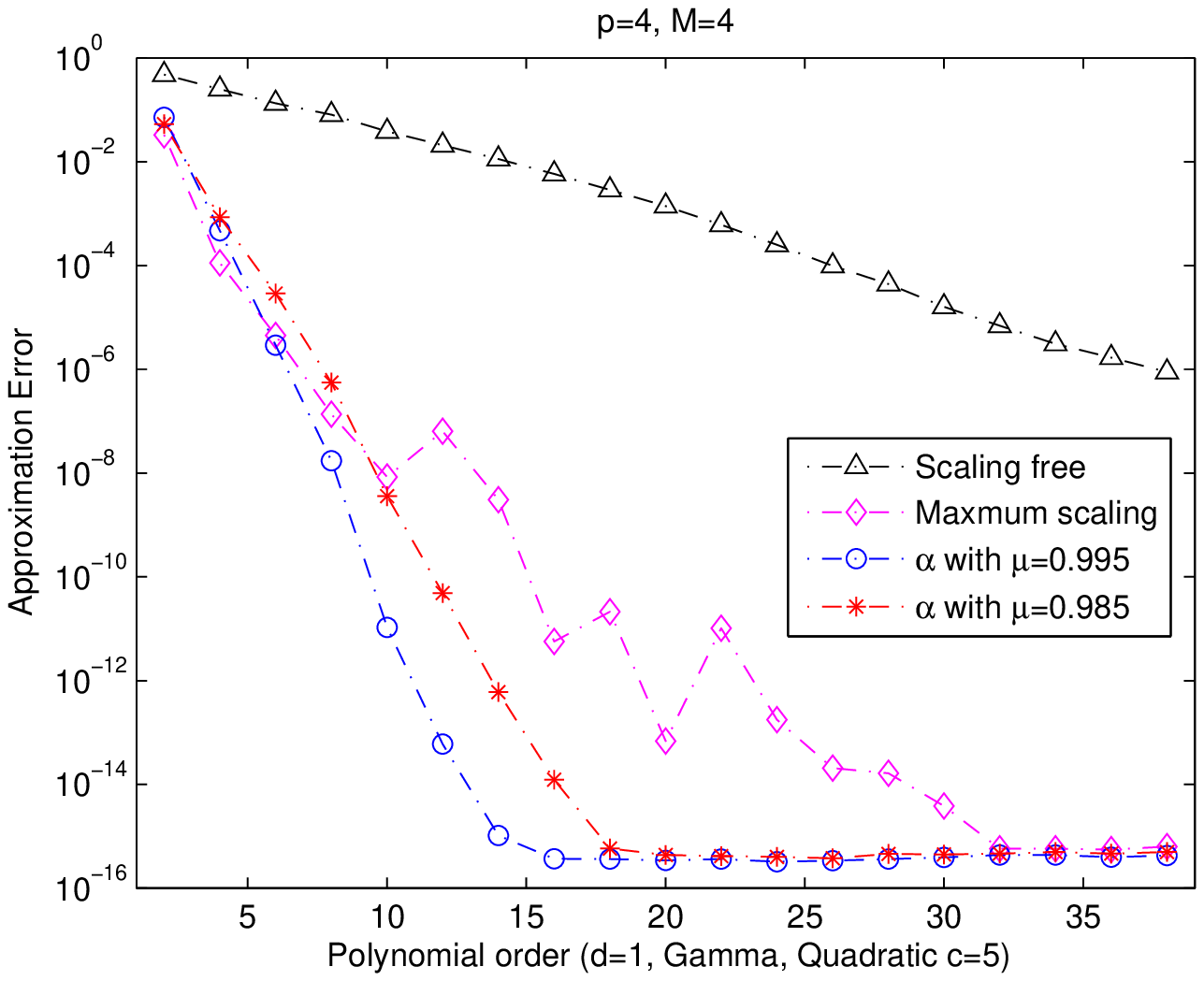}
\includegraphics[width=0.45\textwidth]{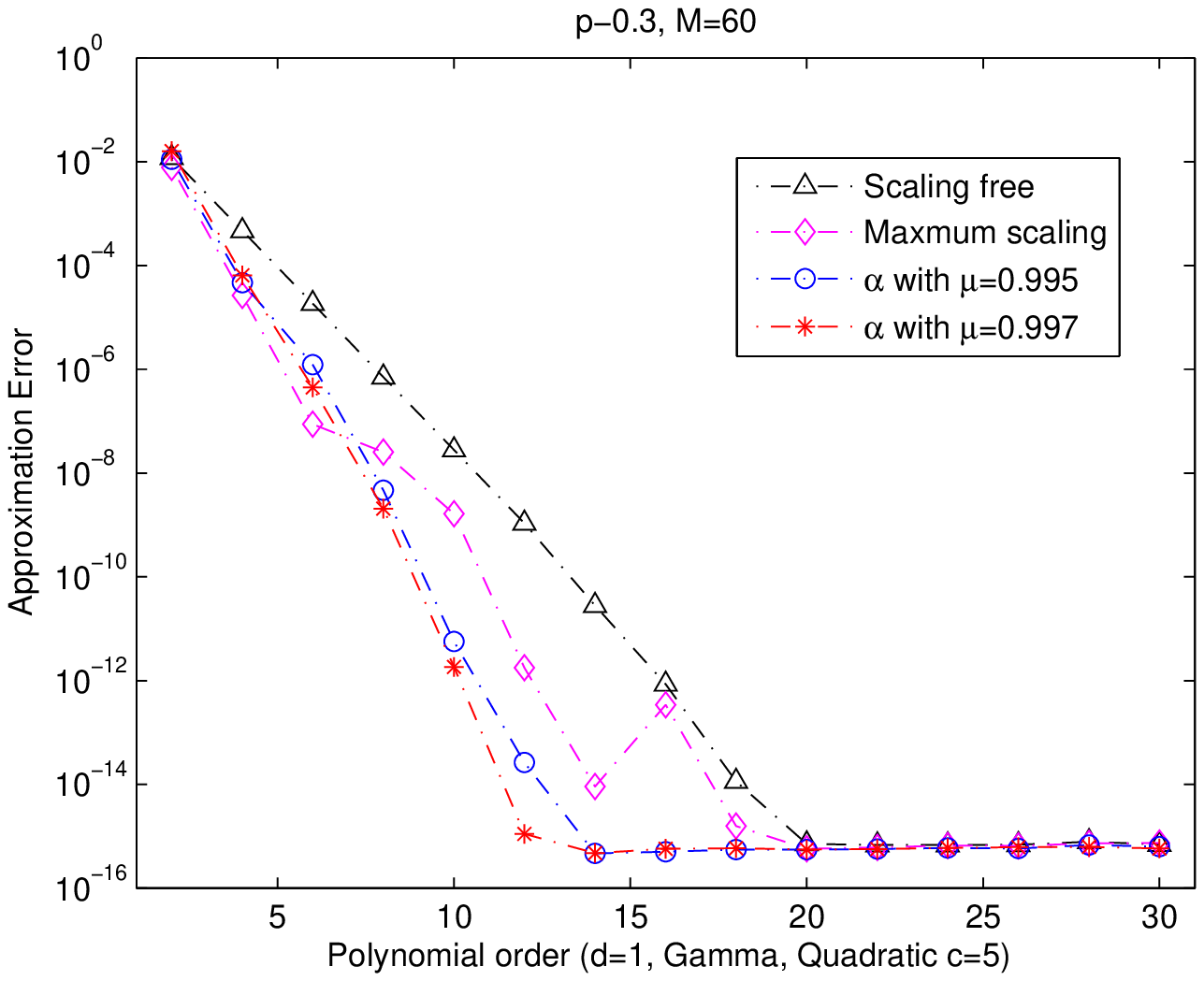}
\end{center}
\caption{\small Problem (\ref{ode}):
Convergence with respect to polynomial order with $m=5*(\# \Lambda)^2$. Left: $\beta=2$, with different scaling $\alpha.$ Right: $\beta=-0.65$, with different scaling $\alpha.$}
\label{fig:lagfun_scaling}
\end{figure}

It is seen from the above example that for problems with Gamma random parameters the \textit{maxmum scaling} can be used.
Moreover, if the partial maximum scaling associated
with parameter $\mu$ is used, then larger $\mu$ (say $\mu=0.995$) should be used.
This is quite different with the Gaussian case.

\subsection{Elliptic problems with lognormal random input}

We now take the following elliptic problems with lognormal random
input as an example
\begin{align}\label{eq:lognormal}
&-\nabla \cdot (a(x,\omega) \nabla u)=f,\quad x\in D,\,\, \,\omega \in \Omega,\nn\\
&u(x,\omega)|_{\partial D}=0.
\end{align}
The coefficient $a(x,\omega): \vec{D}\times \Omega \rightarrow \mathbb{R}$ is a lognormal random field, i.e.,
\begin{equation}
a(x,\omega)=\textmd{e}^{\gamma(x,\omega)}, \quad \gamma(x,\omega) \sim N(\mu,\sigma^2), \quad \forall x\in D,
\end{equation}
where $N(\eta,\sigma^2)$ denotes a Gaussian probability distribution with expected value
$\eta$ and variance $\sigma^2,$ and $\gamma(x,\omega): D \times\Omega \rightarrow \mathbb{R}$ is such that for $x, x' \in D$ the
covariance function $C_\gamma(x, x') = \mathbb{C}\textmd{ov} [\gamma(x, \cdot)\gamma(x', \cdot)]$ depends only on the distance
$||x - x'||$ (isotropic property). Moreover, $C_\gamma(x, x') = C_\gamma(||x - x||)$ is Lipschitz continuous, and is a positive definite function.

Several types of the covariance function $C_\gamma$ have been proposed in the literature \cite{lognormal}. Such as the exponential correlation function
$$C_\gamma(x, x')=\sigma^2 \textmd{exp}
\left(- \frac{||x-x'||_1}{L_c^2}\right),$$
and the Gaussian function
$$C_\gamma(x, x')=\sigma^2 \textmd{exp}
\left(- \frac{||x-x'||^2}{L_c^2}\right),$$
where $L_c> 0$ is called correlation length.

The well-posedness of the lognormal problem (\ref{eq:lognormal}) has been investigated
in \cite{Wellpose}. The optimal convergence rate of its gPC approximation and QMC approach has also been analyzed theoretically in \cite{Schwab,QMC_log}. To solve the problem, one first transform the original problem into a finite model, by means of the Karhunen-L$\grave{\mathrm{o}}$eve expansion:
\begin{equation}\label{eq:kl}
a^N(x,\omega) \approx \bar{a}(x)+ \sum_{i=1}^N \sqrt{\lambda_i} y_i(\omega)a_i(x),
\end{equation}
where $\{\lambda_i\}_{i=1}^\infty$ and $\{a_i\}_{i=1}^\infty$
are the eigenvalues and orthogonal eigenfunctions of
$C_\gamma(x, x')$ i.e.,
\begin{align*}
\int_{D\times D}
C_\gamma(x, x')a_i(x)dx=\lambda_ia_i(x').
\end{align*}
Apart from using (\ref{eq:kl}), other techniques such as Fourier expansion \cite{lognormal} can be used.

\begin{figure}[h]
\begin{center}
\includegraphics[width=0.45\textwidth]{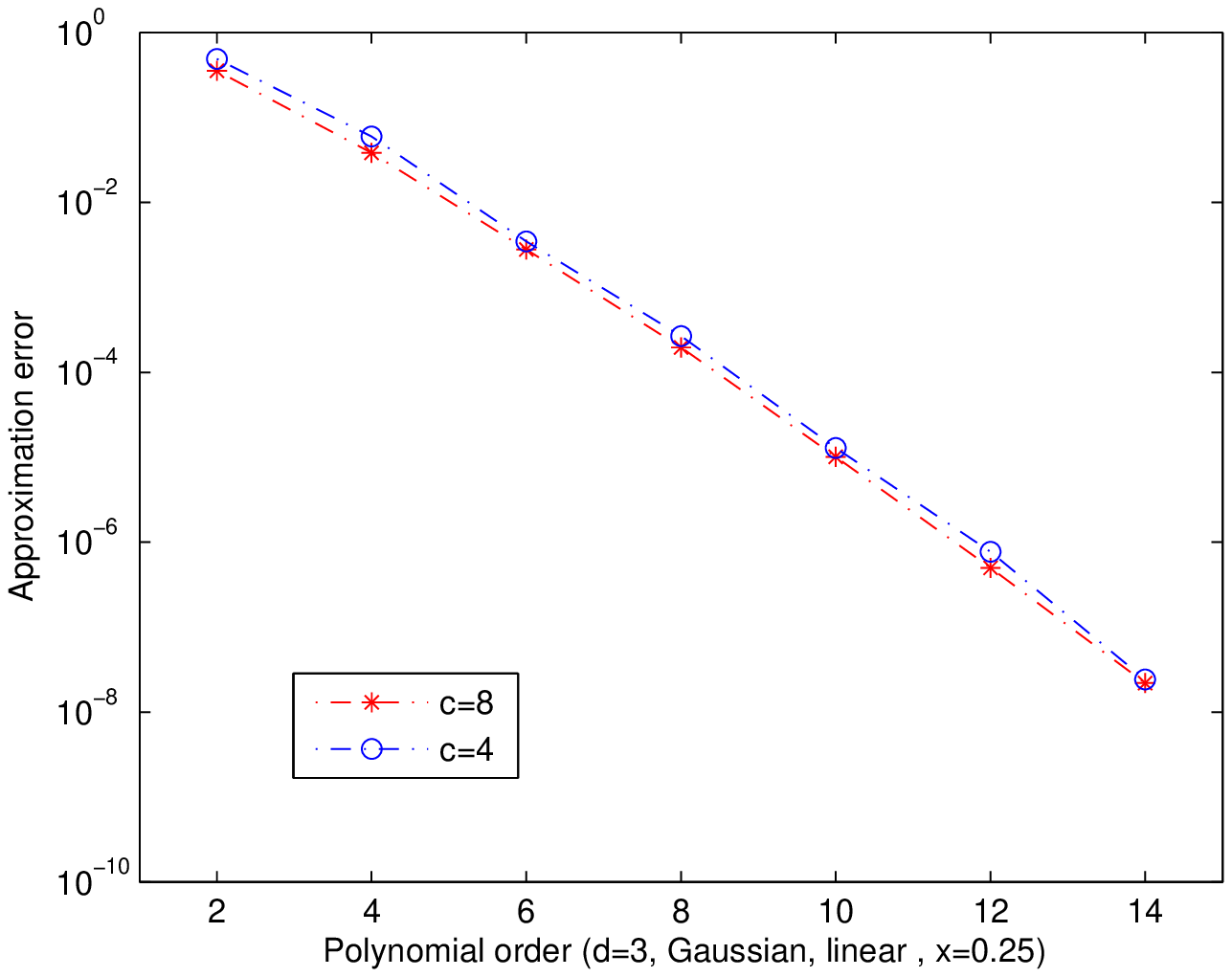}
\includegraphics[width=0.45\textwidth]{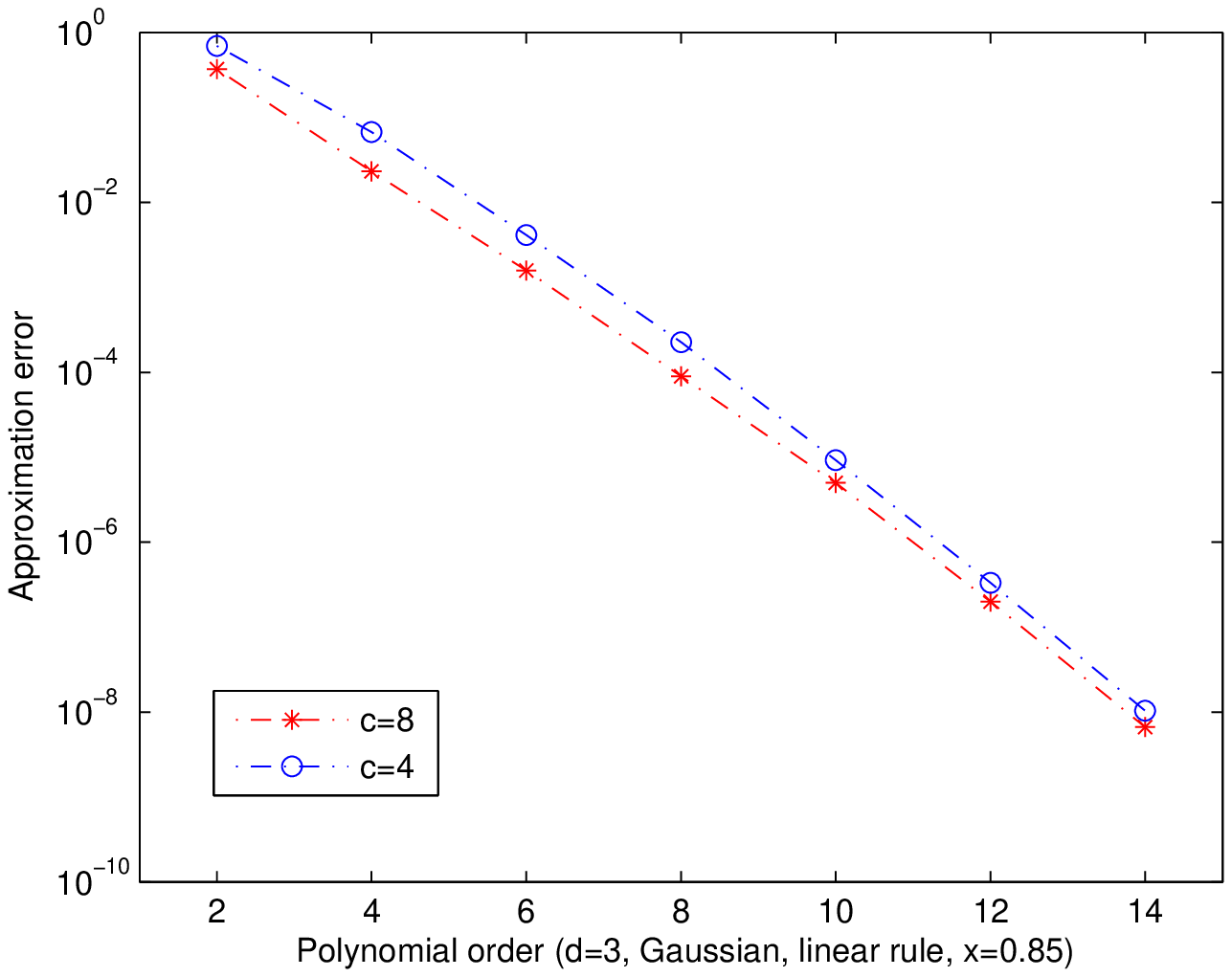}
\end{center}
\caption{Problem (\ref{eq:lognormal}) with random coefficient
(\ref{eq:lognormal_example}): Convergence with respect to polynomial order with the linear rule $m=c*(\# \Lambda).$ Left: $x_0=0.25.$ Right: $x_0=0.85.$}\label{fig:elliptic}
\end{figure}

Here, we consider the least squares approach to obtain the QoI of problem (\ref{eq:lognormal}) with finite parameters random coefficient (\ref{eq:kl}). Let us have a close observation to the following simple case:
\begin{equation}
\label{4e5}
-\nabla \cdot (e^{c y} \nabla u)=\sin(\pi x),
\end{equation}
where $y$ is a Gaussian random variable and $p$ is a constant.
The exact solution is $u=e^{-cy}{\sin(\pi x)}/{\pi^2}.$ In our least
square framework,  we wish to expand the function
\[
\tilde{u}=e^{-\frac{y^2}{2}}u= e^{-\frac{(y+c)^2}{2}}
e^{\frac{c^2}{2}}\frac{\sin(\pi x)}{\pi^2} ,
\]
 which admits similar decay property as the density function
$e^{-{y^2}/{2}}.$  Consequently, a scaling is {\em not} needed
and the standard Hermite {\em function} approximation
without scaling  should work.
In fact, it follows from the maximum principle that the solutions of
(\ref{4e5}) or (\ref{eq:kl})) are bounded. Therefore, Hermite function
approach without scaling should work well.

We now consider problem (\ref{eq:lognormal}) with the following random coefficient
\begin{equation}\label{eq:lognormal_example}
a^N(x,\omega) = y_0+ \frac{1}{2}\Big( y_1\cos(\pi x) + y_2 \sin(\pi x)\Big),   \quad x\in [0,1]
\end{equation}
with $y_i \sim N(0,1), i=0,1,2.$ That is, three Gaussian parameters are used.  We believe that the exact solution of this problem has Gaussian decay profile similar to the above
simple illustration, and we will use the least square approximation with the non-scaling Hermite function approach.
Suppose we are interested in the QoI:
\begin{equation}
{\rm QoI}= \int_{\Gamma} e^{-\frac{\mathbf{y}^2}{2}} u^2(x_0,
\mathbf{y}) d\mathbf{y}.
\end{equation}
 In the computations, the elliptic equations are solved by standard finite element method. As the exact solution is not available, we use a high level sparse grid collocation method to obtain the reference solution. The numerical error using the least square approximation with respect to the bases order are shown in Fig. \ref{fig:elliptic}. The linear rule $m=c*(\# \Lambda)$ is used, and different $x_0$ are considered. As discussed above, Hermite function approach without scaling indeed works well; even the linear rule gives very good convergence rate.

We close this section by pointing out that
only two illustrative examples are provided to demonstrate
the performance of the least squares approximation with
Hermite (Laguarre) functions for solving the UQ problems.
In fact, practical problems in UQ can be very complicated,
and we may need to solve problems with very high dimensional
parameters. An alternative way to handle high dimensional
problems is to use the $L^1$ minimization framework \cite{L1}
instead of the least squares approach. However,
such framework relies on the assumption that the
solution admits certain \textit{sparse} structure,
and this will be part of our future studies.

\section{Conclusions}

In this paper, we investigate the problem of approximating
multivariate functions in unbounded domains by using
discrete least-squares projection with random points evaluations.
We first demonstrate that the traditional Hermite (Laguerre)
polynomials chaos expansion suffers from the numerical
instability in the sense that \textit{unpractical} number of points, i.e. $(\#\Lambda)^{c\#\Lambda}$, is  needed to guarantee the
stability in the least squares framework.
To improve this, we propose to use the Hermite (Laguerre)
{\em functions} approach. Then the mapped uniformly distributed random points are used to control the condition number of the design matrices.
It is demonstrated that with the Hermite (Laguerre)
functions approach the stability can be much improved,
even if the number of design points scales
\textit{linearly} with the dimension of
the approximation space. On the other hand,
for problems involving exponential decay
the convergence may be very slow due to the poor conrvergence
property of the Hermite (Laguerre) polynomial/function approach.
To improve this, scaling factors are investigated
to accelerate the convergence rate. This is particularly useful
if the underlying function to be approximated decay
much faster or much slower than that of the Gaussian (Gamma)
measure. A principle for choosing the quasi-optimal scaling factor
is provided.  Applications to parametric UQ problems
are illustrated.

We emphasize that for approximating
multivariate functions in unbounded domains by using
discrete least-squares projection two parameters are involved:
one is the transformation parameter $L$ in (\ref{eq:mapping}), and another is the scaling factor $\alpha$ in (\ref{eq:herfun_p=6}).
The transformation parameter $L$ is used to control the stability
while the scaling factor $\alpha$ is used to control the rate of
convergence. In this work, as the sample points in the least
square approach are generated randomly, an idea of dropping bad points is used, which lead to a useful
formula (\ref{eq:quasi-optimal}).

There are, however, a number of important issues
deserving further attention, which are listed below.

\begin{itemize}
\item {\em Optimal mapping}.
In this work, we used a class of mapping (\ref{eq:mapping})
to transform the uniform random points in a bounded domain
to unbounded domains, where a parameter $L$ is used to
control the condition number.
Are there better mappings that even work well with a linear rule?

\item {\em Optimal scaling.}
The scaling factor $\alpha$ given in Section 3.2 is determined by
the size of the effective support, i.e., $M$. If the data
information is sufficiently large then $M$ can be easily obtained.
In the UQ problems large data
information means a significant amount of computational time
for solving differential equations.
One possible remedy is to use less accurate but fast
(even parallel) solvers, as a
rough $M$ should serve the purpose. This remains to be examined.

\item{\em The correlation of mapping and scaling.}
 Is there any correlation between $L$ and the scaling factor
$\alpha$?

\item {\em High dimensions}.
If the underlying solution admits certain \textit{sparsity} structure, we may use the $L^1$ minimization framework instead of the least-squares approach to further enhance
the computational efficiency.
This topic with suitable transformation and scaling should be
studied.
\end{itemize}

\section*{Acknowledgment}
 The research of T. Tang is supported by Hong Kong Research
Grants Council (RGC), Hong Kong Baptist University, and an
NSFC-RGC joint research grant. The research of T. Zhou is supported by the National Natural Science Foundation of China (No.91130003 and No.11201461).

\end{document}